\documentclass[11pt,reqno]{amsart}
\usepackage{amsmath,amssymb,mathrsfs}
\usepackage{graphicx,cite}
\usepackage{color}
\usepackage{subfigure}
\usepackage{bm}

\newtheorem{theorem}{Theorem}[section]

\newtheorem{lemma}[theorem]{Lemma}

\newtheorem{assumption}{Assumption}
\newtheorem{definition}{Definition}

\numberwithin{equation}{section}

\allowdisplaybreaks[4]

\begin{document}

\title[Inverse random source problems]{Inverse source problems for the stochastic wave equations: far-field patterns}

\author{Jianliang Li}
\address{School of Mathematics and Statistics, Yunnan University, Kunming 650091, China.}
\email{lijl@amss.ac.cn}

\author{Peijun Li}
\address{Department of Mathematics, Purdue University, West Lafayette, Indiana
47907, USA.}
\email{lipeijun@math.purdue.edu}

\author{Xu Wang}
\address{LSEC, ICMSEC, Academy of Mathematics and Systems Science, Chinese Academy of Sciences, Beijing 100190, China, and School of Mathematical Sciences, University of Chinese Academy of Sciences, Beijing 100049, China.}
\email{wangxu@lsec.cc.ac.cn}

\thanks{The first author was supported by the NNSF of China under Grant 12171057. The second author was supported in part by the NSF grant DMS-1912704.}

\subjclass[2010]{35R30, 35R60, 60H15}

\keywords{inverse source problem, stochastic wave equation, Gaussian random field, pseudo-differential operator, far-field pattern, uniqueness}

\begin{abstract}
This paper addresses the direct and inverse source problems for the stochastic acoustic, biharmonic, electromagnetic, and elastic wave equations in a unified framework. The driven source is assumed to be a centered generalized microlocally isotropic Gaussian random field, whose covariance and relation operators are classical pseudo-differential operators. Given the random source, the direct problems are shown to be well-posed in the sense of distributions and the regularity of the solutions are given. For the inverse problems, we demonstrate by ergodicity that the principal symbols of the covariance and relation operators can be uniquely determined by a single realization of the far-field pattern averaged over the frequency band with probability one.
\end{abstract}

\maketitle

\section{Introduction}

Inverse scattering problems are to determine the nature of scatterers from a knowledge of the wave field. They have played an essential role in many scientific areas such as radar and sonar, geophysical exploration, medical imaging, and nondestructive testing. These problems are challenging due to the ill-posedness and nonlinearity \cite{CK19}. In many situations, it is desirable to 
describe the scatterer as a random field in order to handle uncertainties of the surrounding environment. Compared with deterministic counterparts, stochastic scattering problems have substantially more difficulties because of two additional obstacles: the scatterer is sometimes too rough to exist point-wisely and should be understood in the sense of distributions instead; the randomness makes it meaningless and impossible to characterize the scatterer by a particular realization. As a result, the statistics, such as mean and variance, of the random scatterer are used to quantify the uncertainties of the scatterer and are of more interest in stochastic inverse scattering problems. Recently, stochastic inverse scattering problems have attracted great attention, and many new results are available for various problems, such as random medium problems \cite{LW3}, random potential problems \cite{CHL,LLW1,LLW2,LPS08, LLM2}, random impedance problems \cite{HLP}, and random surface problems \cite{BLX, FLN, HSSS}. 

As an important research subject in inverse scattering theory, the inverse random source problem has been extensively studied. When the source is modeled by an additive white noise, the mean and standard deviation of the source can be reconstructed from the statistics of the wave field \cite{BCL16,BCL18,BX,GX, LLM1}. In these approaches, the near-field scattering data needs to be measured for a fairly large number of realizations of the random source. Motivated by \cite{LPS08}, a new model is developed for the random source, which is assumed to be a real-valued generalized microlocally isotropic Gaussian (GMIG) random field with its covariance operator being a classical pseudo-differential operator. It is shown that the principal symbol of the covariance operator can be uniquely determined by the amplitude of the near-field scattering data averaged over the frequency band, generated by a single realization of the random source, see \cite{LHL,LW1} for acoustic waves, \cite{LHL,LL} for elastic waves, and \cite{LW4} for biharmonic waves. The inverse random source problem for electromagnetic waves is considered in \cite{LW2}, where the source is modeled by a complex-valued centered GMIG random field whose real and imaginary parts are assumed to be independent and identically distributed, leading to the relation operator being zero. The uniqueness result states that the high frequency limit of the variation of the electric field can uniquely determine the principal symbol matrix of the covariance operator for the random source. Moreover, by means of ergodicity in the frequency domain, the amplitude of the electric field averaged over the frequency band, obtained from a single path of the random source, can uniquely determine the diagonal entries of the principal symbol matrix. 

In this work, we intend to examine the direct and inverse source problems for the stochastic acoustic, biharmonic, electromagnetic, and elastic wave equations in a unified framework by using the far-field patterns. There are two main contributions: 

\begin{enumerate}
 \item the well-posedness of the direct problems are established for more general random sources; 
 
 \item the uniqueness of the inverse problems are obtained for both the covariance and relation operators. 
\end{enumerate}

Specifically, we consider the four commonly encountered wave equations, i.e., the Helmholtz equation, the biharmonic wave equation, Maxwell's equations, and the Navier equation. The driven source is assumed to be a complex-valued centered GMIG random field whose covariance operator and relation operator are classical pseudo-differential operators, which removes the limitation that the real and imaginary parts are independent and identically distributed. As is shown in the context, this type of sources is too rough to exist pointwisely and should be understood as distributions. Given such rough sources, the direct problems are shown to be well-posed and regularity of the solutions is also obtained. For the inverse problems, we demonstrate that the principal symbol matrices of the covariance and relation operators can be uniquely determined by the high frequency limit of the correlation of the far-field pattern. Moreover, with the aid of ergodicity of the far-field pattern in the frequency domain, the uniqueness is established for the principal symbol matrices of the covariance and relation operators with respect to the far-field pattern obtained from a single realization of the random source almost surely. 

The paper is organized as follows. In Section 2, some preliminaries are given for the high dimensional complex-valued GMIG random fields and
the fundamental solutions to the Helmholtz equation, the biharmonic wave equation, and the Navier equation.
Sections 3--6 are devoted to the direct and inverse random source problems for acoustic waves, biharmonic waves, electromagnetic waves, and elastic waves, respectively. The direct problems are examined and the uniqueness of the inverse problems is addressed. The paper concludes with some general remarks and discussions on the future work in Section 7.

\section{Preliminaries}

In this section, we introduce the $\mathbb C^n$-valued $(n\in\mathbb N)$ GMIG random fields and the fundamental solutions to the Helmholtz, biharmonic wave, and Navier equations. 

\subsection{$\mathbb C^n$-valued GMIG random fields}

Let $\mathcal O\subset\mathbb R^d$ be an open domain. Denote by $C_0^{\infty}(\mathcal O;\mathbb{F})$ the set of $\mathbb{F}$-valued smooth functions with compact supports contained in $\mathcal O$, where $\mathbb{F}$ stands for the real-valued space $\mathbb{R}$, the complex-valued space $\mathbb{C}$, or the $n$-dimensional complex-valued space $\mathbb{C}^n$. Define the space of test functions by $\mathcal{D}(\mathcal O;\mathbb{F})$, which is $C_0^{\infty}(\mathcal O;\mathbb{F})$ equipped with a locally convex topology. The dual space $\mathcal{D'}(\mathcal O;\mathbb{F})$ of  $\mathcal{D}(\mathcal O;\mathbb{F})$ is the space of distributions on $\mathcal O$ with a weak-star topology.

Denote by $W^{\gamma,q}(\mathcal O;\mathbb F)$ the $\mathbb F$-valued classical Sobolev spaces with $\gamma\in\mathbb R$ and $q\in(1,\infty)$, and by $W_0^{\gamma,q}(\mathcal O;\mathbb F)$ the closure of $C_0^\infty(\mathcal O;\mathbb F)$ in $W^{\gamma,q}(\mathcal O;\mathbb F)$ with $\gamma>0$. For simplicity, the domain $\mathbb F$ will be omitted if $\mathbb F=\mathbb C$ or $\mathbb C^n$, i.e., $W^{\gamma,q}(\mathcal O)=W^{\gamma,q}(\mathcal O;\mathbb C)$ if $\mathbb F=\mathbb C$ and $\bm W^{\gamma,q}(\mathcal O)=W^{\gamma,q}(\mathcal O;\mathbb C^n)$ if $\mathbb F=\mathbb C^n$. 

Let $(\Omega,\mathcal{F},\mathbb{P})$ be a complete probability space, where $\Omega$ is a sample space, $\mathcal{F}$ is a $\sigma$-algebra on $\Omega$, and $\mathbb{P}$ is a probability measure on $(\Omega,\mathcal{F})$. 

First we consider scalar fields when $n=1$.
A scalar field $f$ is said to be a $\mathbb C$-valued generalized Gaussian random field if $f:\Omega\to \mathcal{D'}(\mathcal O)$ is a distribution satisfying that, for each $\omega\in\Omega$, the path $f[\cdot](\omega)$ is a linear functional on $\mathcal{D}(\mathcal O)$ and for any test function $\psi\in \mathcal{D}(\mathcal O)$,
$
f[\psi]=\langle f,\psi\rangle:\Omega\to\mathbb C
$
is a $\mathbb C$-valued Gaussian random variable.

Let $\mathcal O=\mathbb R^d$. The $\mathbb C$-valued generalized Gaussian random field $f$ defined on $\mathbb R^d$ is uniquely determined by its expectation ${\mathbb E}f\in \mathcal{D'}(\mathbb R^d)$, covariance operator $\mathcal C_f:\mathcal{D}(\mathbb R^d)\to \mathcal{D'}(\mathbb R^d)$, and relation (pseudo-covariance) operator 
$\mathcal R_f:\mathcal{D}(\mathbb R^d)\to \mathcal{D'}(\mathbb R^d)$, which are defined by
\begin{eqnarray*}\label{a2}
\langle {\mathbb E}f,\psi\rangle :&=& {\mathbb E}\langle f, \psi \rangle,\\\label{a5}
\langle\mathcal C_f\varphi, \psi\rangle:&=&\mathbb E\big[\langle\overline{ f-\mathbb Ef},\varphi\rangle\langle f-\mathbb Ef,\psi\rangle\big]=\left\langle\mathbb E\big[(f-\mathbb Ef)\otimes(\overline{f-\mathbb Ef})\big],\psi\otimes\varphi\right\rangle,\\\label{a6}
\langle\mathcal R_f\varphi, \psi\rangle:&=&\mathbb E\left[\langle f-\mathbb Ef,\varphi\rangle\langle f-\mathbb Ef,\psi\rangle\right]=\left\langle\mathbb E\big[(f-\mathbb Ef)\otimes(f-\mathbb Ef)\big],\psi\otimes\varphi\right\rangle
\end{eqnarray*}
for any $\varphi,\psi\in \mathcal{D}(\mathbb R^d)$.
It is easy to note that $\mathcal C_f=\mathcal R_f$ if $f$ is $\mathbb R$-valued.

Introduce the space of symbols of order $-m$:
\begin{eqnarray*}
\mathcal{S}^{-m}(\mathbb R^d\times\mathbb R^d):=\Big\{\sigma\in C^{\infty}(\mathbb R^d\times\mathbb R^d):|\partial_{\xi}^{\gamma_1}\partial_x^{\gamma_2}\sigma(x,\xi)|\le C_{\gamma_1,\gamma_2}(1+|\xi|)^{-m-|\gamma_1|}\Big\},
\end{eqnarray*}
where $C_{\gamma_1,\gamma_2}$ is a positive constant depending on $\gamma_1$ and $\gamma_2$.

\begin{definition}\label{def1}
A $\mathbb C$-valued generalized Gaussian random field $f$ on $\mathbb R^d$ is said to be microlocally isotropic of order $-m$ in $D$ if its covariance and relation operators $\mathcal C_f$ and $\mathcal R_f$ are classical pseudo-differential operators of order $-m$, whose symbols $\sigma^c,\sigma^r\in\mathcal S^{-m}(\mathbb R^d\times\mathbb R^d)$ satisfy 
\[
\sigma^\eta(x,\xi)=a^\eta(x)|\xi|^{-m}+b^\eta(x,\xi),
\]
where $b^\eta\in\mathcal S^{-m-1}(\mathbb R^d\times\mathbb R^d)$ and $a^\eta,b^\eta(\cdot,\xi)\in C_0^\infty(D)$ for $\eta\in\{c,r\}$.
\end{definition}

Using the definition of pseudo-differential operators, we get
\begin{eqnarray}\label{a9}
\mathcal C_f\varphi(x)&=&\frac{1}{(2\pi)^d}\int_{{\mathbb R}^d}e^{{\rm i}x\cdot \xi}\sigma^c(x,\xi)\hat{\varphi}(\xi)d\xi,\\\label{a10}
\mathcal R_f\varphi(x)&=&\frac{1}{(2\pi)^d}\int_{{\mathbb R}^d}e^{{\rm i}x\cdot \xi}\sigma^r(x,\xi)\hat{\varphi}(\xi)d\xi,
\end{eqnarray}
where 
\[
\hat{\varphi}(\xi)=(\mathscr{F}\varphi)(\xi):=\int_{\mathbb R^d}e^{-{\rm i}x\cdot\xi}\varphi(x)dx
\]
denotes the Fourier transform of $\varphi$. By the Schwartz kernel theorem, there exist unique kernels $K_f^c,K_f^r\in\mathcal D'(\mathbb R^d\times\mathbb R^d)$ such that
\begin{eqnarray*}
\langle \mathcal C_f\varphi,\psi\rangle=\langle K_f^c,\psi\otimes\varphi\rangle,\quad
\langle \mathcal R_f\varphi,\psi\rangle=\langle K_f^r,\psi\otimes\varphi\rangle,
\end{eqnarray*}
which imply that
\begin{eqnarray*}
K_f^c(x,y)&=&{\mathbb E}\big[(f(x)-{\mathbb E}f(x))(\overline{f(y)-{\mathbb E}f(y)})\big],\\
K_f^r(x,y)&=&{\mathbb E}\big[(f(x)-{\mathbb E}f(x))(f(y)-{\mathbb E}f(y))\big]
\end{eqnarray*}
are distributions in $\mathcal D'(\mathbb R^d\times\mathbb R^d)$. Using \eqref{a9}--\eqref{a10}, we obtain the following bijection between the kernel $K_f^\eta$ with $\eta\in\{c,r\}$ and the symbol $\sigma^\eta$: 
\begin{eqnarray}\label{a12}
K_f^\eta(x,y)=\frac{1}{(2\pi)^d}\int_{\mathbb R^d}e^{{\rm i}(x-y)\cdot \xi}\sigma^\eta(x,\xi)d\xi={\mathscr F}^{-1}(\sigma^\eta(x,\cdot))(x-y).
\end{eqnarray}

Taking the Fourier transform on the both sides of (\ref{a12}) with respect to $x-y$ gives
\begin{eqnarray}\nonumber
&&\int_{\mathbb R^d}\left(\int_{\mathbb R^d}K_f^\eta(x,y)e^{-{\rm i}(x-y)\cdot\xi}dy\right)\varphi(x)dx=\int_{\mathbb R^d}\sigma^\eta(x,\xi)\varphi(x)dx\\\label{a14}
&&=\int_{\mathbb R^d}a^\eta(x)|\xi|^{-m}\varphi(x)dx+\int_{\mathbb R^d}b^\eta(x,\xi)\varphi(x)dx.
\end{eqnarray}

The regularity of random fields given in Definition \ref{def1} depends on the order $-m$. It has been studied in \cite{LW1} and is stated in the following lemma.

\begin{lemma}\label{lem1}
Let $f$ be a $\mathbb C$-valued GMIG random field of order $-m$ in $D$.

(i) If $m\in (d,d+2)$, then $f\in C^{0,\alpha}(D)$ almost surely for all $\alpha\in\left(0,\frac{m-d}{2}\right)$.

(ii) If $m\le d$, then $f\in W^{\frac{m-d}{2}-\epsilon, p}(D)$ almost surely for any $\epsilon>0$ and $p\in (1,\infty)$.
\end{lemma} 

Now let us consider vector fields for $n>1$. A vector field ${\bm f}=(f_1,...,f_n)^{\top}$ is said to be a $\mathbb C^n$-valued GMIG random field of order $-m$ in $D$ if each component $f_j, j=1,...,n,$ is a $\mathbb C$-valued GMIG random field of the same order $-m$ in $D$. Similarly, the $\mathbb C^n$-valued generalized Gaussian random field ${\bm f}$ is uniquely determined by its expectation ${\mathbb E}{\bm f}\in \bm{\mathcal{D}}'(\mathbb{R}^d)$, covariance operator $\mathcal C_{\bm f}:\bm{\mathcal{D}}(\mathbb{R}^d)\mapsto \bm{\mathcal{D}}'(\mathbb{R}^d)$, and relation operator 
$\mathcal R_{\bm f}:\bm{\mathcal{D}}(\mathbb{R}^d)\mapsto \bm{\mathcal{D}}'(\mathbb{R}^d)$. 
The kernels $K_{\bm f}^c,K_{\bm f}^r\in\mathcal D'(\mathbb R^d\times\mathbb R^d;\mathbb C^{n\times n})$ can be formally expressed as the following distributions:
\begin{eqnarray*}\label{a18}
K_{\bm f}^c(x,y)&=&{\mathbb E}\big[({\bm f}(x)-{\mathbb E}{\bm f}(x))(\overline{{\bm f}(y)-{\mathbb E}{\bm f}(y)})^\top\big],\\\label{a19}
K_{\bm f}^r(x,y)&=&{\mathbb E}\big[({\bm f}(x)-{\mathbb E}{\bm f}(x))({\bm f}(y)-{\mathbb E}{\bm f}(y))^{\top}\big].
\end{eqnarray*}

If $\bm f$ is microlocally isotropic of order $-m$, then there exist symbols $\Sigma^r,\Sigma^c\in\mathcal S^{-m}(\mathbb R^d\times\mathbb R^d;\mathbb C^{n\times n})$ of the form
\[
\Sigma^\eta(x,\xi)=A^\eta(x)|\xi|^{-m}+B^\eta(x,\xi)
\]
with $B^\eta\in\mathcal S^{-m-1}(\mathbb R^d\times\mathbb R^d;\mathbb C^{n\times n})$ and $A^\eta,B^\eta(\cdot,\xi)\in C_0^{\infty}(D;\mathbb C^{n\times n})$
such that
\begin{eqnarray*}
\mathcal C_{\bm f}{\bm\varphi}(x)=\frac{1}{(2\pi)^d}\int_{{\mathbb R}^d}e^{{\rm i}x\cdot \xi}\Sigma^c(x,\xi)\widehat{\bm\varphi}(\xi)d\xi,\quad
\mathcal R_{\bm f}{\bm\varphi}(x)=\frac{1}{(2\pi)^d}\int_{{\mathbb R}^d}e^{{\rm i}x\cdot \xi}\Sigma^r(x,\xi)\widehat{\bm\varphi}(\xi)d\xi
\end{eqnarray*}
for any $\bm\varphi\in\bm{\mathcal D}(\mathbb R^d)$,
and 
\[
K_{\bm f}^\eta(x,y)=\frac1{(2\pi)^d}\int_{\mathbb R^d}e^{{\rm i}(x-y)\cdot\xi}\Sigma^\eta(x,\xi)d\xi=\mathscr F^{-1}(\Sigma^\eta(x,\cdot))(x-y),
\]
where $\eta\in\{c,r\}$. It can also be verified that the kernel $K_{\bm f}^\eta$ satisfies
\begin{eqnarray}\label{a22}
&&\int_{\mathbb R^d}\left(\int_{\mathbb R^d}K_{\bm f}^{\eta}(x,y)e^{-{\rm i}(x-y)\cdot\xi}dy\right){\bm\varphi}(x)dx=\int_{\mathbb R^d}\Sigma^\eta(x,\xi){\bm\varphi}(x)dx\notag\\
&&=\int_{\mathbb R^d}A^{\eta}(x)|\xi|^{-m}{\bm\varphi}(x)dx+\int_{\mathbb R^d}B^{\eta}(x,\xi){\bm\varphi}(x)dx.
\end{eqnarray}

\subsection{The fundamental solutions}

In this subsection, we introduce the fundamental solutions and their asymptotic behaviors of large arguments for the wave equations considered in this work. They play an important role in the analysis. 
 
The fundamental solution of the Helmholtz equation in $\mathbb R^d$ is given by 
\begin{equation*}
	\Phi_d(x,y,\kappa)=\left\{\begin{aligned}\frac{\rm i}{4}H_0^{(1)}(\kappa |x-y|), &&d=2,\\
	\frac{e^{{\rm i}\kappa |x-y|}}{4\pi |x-y|}, &&d=3,\end{aligned}\right.
\end{equation*}
where $\kappa>0$ is the wave number and $H_0^{(1)}$ is the Hankel function of the first kind with order zero. 
Let $\hat{x}:=x/|x|\in \mathbb S^{d-1}$.
Noting (cf. \cite[Theorem 2.6]{CK19})
\[
|x-y|=\sqrt{|x|^2-2|x|\hat{x}\cdot y+|y|^2}=|x|-\hat{x}\cdot y+O\left(|x|^{-1}\right),\quad|x|\to\infty
\]
and (cf. \cite[$(3.105)$]{CK19})
\[
H_0^{(1)}(z)=e^{-{\rm i}\frac{\pi}4}\sqrt{\frac2{\pi z}}e^{{\rm i}z}\left(1+O(|z|^{-1})\right),\quad|z|\to\infty,
\]
we have
\[
\frac{e^{{\rm i}\kappa|x-y|}}{|x-y|}=\frac{e^{{\rm i}\kappa|x|}}{|x|}\left(e^{-{\rm i}\kappa\hat{x}\cdot y}+O\left(|x|^{-1}\right)\right),\quad |x|\to\infty
\]
and
\[
H_0^{(1)}(\kappa|x-y|)=(-{\rm i})e^{{\rm i}\frac{\pi}4}\sqrt{\frac2{\pi\kappa|x|}}e^{{\rm i}\kappa|x|}\left(e^{-{\rm i}\kappa\hat{x}\cdot y}+O\left(|x|^{-1}\right)\right),\quad|x|\to\infty,
\]
which imply  
\begin{eqnarray}\label{a25}
\Phi_d(x,y,\kappa)=\frac{e^{{\rm i}\kappa |x|}}{|x|^{\frac{d-1}{2}}}\left(C_d\kappa^{\frac{d-3}2}e^{-{\rm i}\kappa\hat{x}\cdot y}+O\left(|x|^{-1}\right)\right),\quad |x|\to\infty,
\end{eqnarray}
where 
\begin{eqnarray}\label{a26}
C_d=\left\{\begin{aligned}\frac{e^{{\rm i}\frac{\pi}{4}}}{\sqrt{8\pi}}, &&d=2,\\
\frac{1}{4\pi}, &&d=3.\end{aligned}\right.
\end{eqnarray}

The fundamental solution of the biharmonic wave equation is (cf. \cite{TH17, TS18, LW4})
\begin{eqnarray*}\label{a27}
F_d(x,y,\kappa)=\frac{1}{2\kappa^2}\left[\Phi_d(x,y,\kappa)-\Phi_d(x,y,{\rm i}\kappa)\right].
\end{eqnarray*}
It follows from \eqref{a25} that we have  
\begin{eqnarray}\label{a29}
F_d(x,y,\kappa)=\frac{e^{{\rm i}\kappa |x|}}{|x|^{\frac{d-1}{2}}}\left(\frac{C_d}2\kappa^{\frac{d-7}2}e^{-{\rm i}\kappa\hat{x}\cdot y}+O\left(|x|^{-1}\right)\right),\quad |x|\to\infty.
\end{eqnarray}

For the elastic wave equation, it is strongly elliptic if the Lam\'e parameters $\lambda$ and $\mu$ satisfy $\mu>0$ and $\lambda+2\mu>0$ (cf. \cite[Section 10.4]{M}). Its Green tensor is given by 
\begin{eqnarray*}
\bm{G}_d(x,y,\omega)=\frac{1}{\mu}\Phi_d(x,y,\kappa_{\rm
s})\bm{I}+\frac{1}{\omega^2}\nabla_x\nabla_x^\top\Big[\Phi_d(x,y,
\kappa_{\rm s})-\Phi_d(x,y,\kappa_{\rm p})\Big],
\end{eqnarray*}
where $\omega>0$ is the angular frequency,  $\kappa_{p}:=c_p\omega$ and $\kappa_s:=c_s\omega$ with $c_p=(\lambda+2\mu)^{-\frac{1}{2}}$ and $c_s=\mu^{-\frac{1}{2}}$ denote the compressional and shear wave numbers, respectively, and ${\bm I}$ is the $d\times d$ identity matrix. It is shown in \cite[$(27)-(28)$]{VS05} for $d=2$ and in \cite[$(2.2)$]{CS15} for $d=3$ that $\bm G_d$ has the following asymptotic behavior: 
\begin{eqnarray}\label{a31}
\bm{G}_d(x,y,\omega)&=&\frac{e^{{\rm i}\kappa_{\rm p}|x|}}{|x|^{\frac{d-1}{2}}}C_dc_p^{\frac{d+1}2}\omega^{\frac{d-3}2}\hat{x}\hat{x}^\top e^{-{\rm i}\kappa_{\rm p}\hat{x}\cdot y}\notag\\
&&+\frac{e^{{\rm i}\kappa_{\rm s}|x|}}{|x|^{\frac{d-1}{2}}}C_dc_s^{\frac{d+1}2}\omega^{\frac{d-3}2}(\bm{I}-\hat{x}\hat{x}^\top )e^{-{\rm i}\kappa_{\rm s}\hat{x}\cdot y}+O(|x|^{-\frac{d+1}{2}}),\quad|x|\to\infty,
\end{eqnarray}
where the constant $C_d$ is given in \eqref{a26}.

\section{Acoustic waves}
\label{sec:acoustic}

In this section, we investigate the direct and inverse random source problems for the Helmholtz equation. 

\subsection{The direct problem}

Consider the stochastic Helmholtz equation 
\begin{eqnarray}\label{b1}
\Delta u+\kappa^2u=f\quad{\rm in}~ {\mathbb R}^d. 
\end{eqnarray} 
The wave field $u$ is required to satisfy the Sommerfeld radiation condition
\begin{eqnarray}\label{b2}
\lim_{|x|\to\infty}|x|^{\frac{d-1}{2}}\left(\partial_{|x|} u-{\rm i}\kappa u\right)=0.
\end{eqnarray}
The random source $f$ satisfies the following assumption.

\begin{assumption}\label{as1}
The source $f$ is assumed to be a $\mathbb C$-valued centered GMIG random field of order $-m$ in a bounded domain $D\subset\mathbb R^d$. The principal symbols of its covariance and relation operators have the forms 
$a^c(x)|\xi|^{-m}$ and $a^r(x)|\xi|^{-m}$, respectively, where $a^c, a^r\in C_0^\infty(D)$.
\end{assumption}

The problem (\ref{b1})--(\ref{b2}) was studied in \cite{LHL,LW3}, where $f$ was assumed to be a $\mathbb R$-valued centered GMIG random field of order $-m$ with $m\in(d-1,d]$. When $f$ is $\mathbb C$-valued with its covariance and relation operators being of the same order $-m$, its regularity is the same as the $\mathbb R$-valued case. The well-posedness of \eqref{b1}--\eqref{b2} may be obtained directly based on the results in \cite{LHL,LW3}, but the parameters are not optimal. The following result presents the well-posedness of \eqref{b1}--\eqref{b2}, the parameters are different from the existing results and allow more general and rougher sources. 

\begin{theorem}\label{thm1}
Let $f$ satisfy Assumption \ref{as1} with $m\in (d-4,d]$.
The problem (\ref{b1})--(\ref{b2}) is well-posed in the sense of distributions with a unique solution given by 
\begin{eqnarray}\label{b3}
u(x,\kappa)=-\int_{\mathbb R^d}\Phi_d(x,y,\kappa)f(y)dy,\quad x\in\mathbb R^d,
\end{eqnarray}
where  $u\in W_{\rm loc}^{\gamma,q}({\mathbb R}^d)$ almost surely for any $q>1$ and $$0<\gamma<\min\left\{\frac{4-d+m}2,\frac{4-d+m}2+\left(\frac1q-\frac12\right)d\right\}.$$
\end{theorem}

\begin{proof}
By Assumption \ref{as1}, we have $f\in W^{\frac{m-d}2-\epsilon,p}(D)$ for any $\epsilon>0$ and $p>1$ according to Lemma \ref{lem1}. It follows from the Kondrachov embedding theorem that
\[
W^{\frac{m-d}2-\epsilon,p}(D)\hookrightarrow H^{-s_1}(D)
\]
is continuous for any $p\ge2$ and $s_1\in(\frac{d-m}2,2)$. 

Let $G\subset\mathbb R^d$ be a bounded domain with a locally Lipschitz boundary. Define the volume potential operator $\mathcal H_\kappa$ by
\[
(\mathcal H_\kappa f)(x):=-\int_{\mathbb R^d}\Phi_d(x,y,\kappa)f(y)dy.
\]
Following the same procedure used in \cite[Lemma 3.1]{LW3} yields that $\mathcal H_\kappa:H^{-s_1}(D)\to H^{s_2}(G)$ is bounded for any $s_1,s_2>0$ satisfying $s:=s_1+s_2\in(0,2]$. More precisely, we consider spaces $C^{0,\alpha}(D)$ and $C^{2,\alpha}(G)$ with $\alpha\in(0,1)$ equipped with scalar products
\[
(f_1,f_2)_{C^{0,\alpha}(D)}:=(\tilde f_1,\tilde f_2)_{H^{s_2-2}(\mathbb R^d)}\quad\forall~f_1,f_2\in C^{0,\alpha}(D)
\]
and
\[
(g_1,g_2)_{C^{2,\alpha}(G)}:=(\tilde g_1,\tilde g_2)_{H^{s_2}(\mathbb R^d)}\quad \forall~g_1,g_2\in C^{2,\alpha}(G),
\]
respectively. Here, $\tilde f_i$ and $\tilde g_i$, $i=1,2$, denote the zero extensions of $f_i$ and $g_i$ outside $D$ and $G$, respectively. We then obtain 
\[
\|\mathcal H_\kappa f\|_{H^{s_2}(G)}=\|\mathcal H_\kappa f\|_{C^{2,\alpha}(G)}\lesssim\|f\|_{C^{0,\alpha}(D)}=\|f\|_{H^{s_2-2}(D)}\le\|f\|_{H^{-s_1}(D)}. 
\]
Since $s_1\in(\frac{d-m}2,2)$, it holds $0<s_2\le2-s_1<\frac{4-d+m}2$. Choose $s_2=\frac{4-d+m}2-\epsilon$. Then for $\gamma$ and $q$ satisfying the assumptions in the theorem, there must exist some $\epsilon>0$ such that $\gamma<s_2$ and $\frac1q>\frac12-\frac{s_2-\gamma}d$, and hence the embedding
\[
H^{s_2}(G)\hookrightarrow W^{\gamma,q}(G)
\]
is continuous, which completes the proof.
\end{proof}


\subsection{The inverse problem}

The inverse source problem aims to recover the principal symbols $a^c$ and $a^r$ of the covariance and relation operators, respectively, from the far-field pattern of the wave field. Combining \eqref{a25} and \eqref{b3} gives
\begin{eqnarray*}\label{b5}
u(x,\kappa)=\frac{e^{{\rm i}\kappa |x|}}{|x|^{\frac{d-1}{2}}}\left(u^{\infty}(\hat{x},\kappa)+O(|x|^{-1})\right),\quad |x|\to\infty,
\end{eqnarray*}
where $u^{\infty}$ is known as the far-field pattern and is given by 
\begin{eqnarray}\label{b6}
u^{\infty}(\hat{x},\kappa)=-C_d\kappa^{\frac{d-3}2}\int_{\mathbb R^d} e^{-{\rm i}\kappa\hat{x}\cdot y}f(y)dy.
\end{eqnarray}

First, we show that the Fourier modes of $a^c$ and $a^r$ can be determined by the expectation of the high frequency limit of the far-field pattern, which is stated in the following lemma.

\begin{lemma}\label{lm2}
Let $f$ satisfy Assumption \ref{as1} with $m\in(d-4,d]$. For any $\tau\geq 0$, it holds
\begin{eqnarray}\label{b7}
\lim_{\kappa\to\infty}\kappa^{m+3-d}{\mathbb E}\left[u^{\infty}(\hat{x},\kappa+\tau)\overline{u^{\infty}(\hat{x},\kappa)}\right]&=&|C_d|^2\widehat{a^c}(\tau\hat{x}),\\\label{b8}
\lim_{\kappa\to\infty}\kappa^{m+3-d}{\mathbb E}\left[u^{\infty}(\hat{x},\kappa+\tau)u^{\infty}(-\hat{x},\kappa)\right]&=&C_d^2\widehat{a^r}(\tau\hat{x}).
\end{eqnarray}
\end{lemma}

\begin{proof}
It follows from \eqref{b6} that 
\begin{eqnarray}\label{eq:lm2}
&&{\mathbb E}\left[u^{\infty}(\hat{x},\kappa+\tau)\overline{u^{\infty}(\hat{x},\kappa)}\right]\notag\\
&=&|C_d|^2(\kappa+\tau)^{\frac{d-3}2}\kappa^{\frac{d-3}2}\int_{\mathbb R^d}\int_{\mathbb R^d} e^{-{\rm i}(\kappa+\tau)\hat{x}\cdot y}e^{{\rm i}\kappa\hat{x}\cdot z}{\mathbb E}\left[f(y)\overline{f(z)}\right]dydz\notag\\
&=&|C_d|^2(\kappa+\tau)^{\frac{d-3}2}\kappa^{\frac{d-3}2}\int_{\mathbb R^d}\left[\int_{\mathbb R^d}K_f^c(y,z)e^{-{\rm i}\kappa\hat{x}\cdot (y-z)}dz\right]e^{-{\rm i}\tau\hat{x}\cdot y}dy\notag\\
& =&|C_d|^2(\kappa+\tau)^{\frac{d-3}2}\kappa^{\frac{d-3}2}\left[\int_{\mathbb R^d}a^c(y)e^{-{\rm i}\tau\hat{x}\cdot y}dy|\kappa\hat{x}|^{-m}+\int_{\mathbb R^d}b^c(y,\kappa\hat{x})e^{-{\rm i}\tau\hat{x}\cdot y}dy\right]\notag\\
& =&|C_d|^2\left(\frac{\kappa}{\kappa+\tau}\right)^{\frac{3-d}2}\kappa^{d-3-m}\widehat{a^c}(\tau\hat{x})+O(\kappa^{d-4-m}),
\end{eqnarray}
where we used the relationship between the kernel $K_f^c$ and the symbol $a^c$ given in \eqref{a14}, and the facts that the residual $b^c\in\mathcal S^{-m-1}(\mathbb R^d\times\mathbb R^d)$ satisfies
$|b^c(y,\kappa\hat x)|\lesssim\kappa^{-m-1}$ 
as $\kappa\to\infty$ and $b^c(\cdot,\xi)\in C_0^\infty(D)$ for any $\xi\in\mathbb R^d$. 
Multiplying both sides of \eqref{eq:lm2} by $\kappa^{m+3-d}$ and taking the limit as $\kappa\to\infty$, we obtain 
\begin{align*}
\lim_{\kappa\to\infty}\kappa^{m+3-d}{\mathbb E}\left[u^{\infty}(\hat{x},\kappa+\tau)\overline{u^{\infty}(\hat{x},\kappa)}\right]d\kappa=&~|C_d|^2\widehat{a^c}(\tau\hat x)\lim_{\kappa\to\infty}\left(\frac{\kappa}{\kappa+\tau}\right)^{\frac{3-d}2}\\
=&~|C_d|^2\widehat{a^c}(\tau\hat x),
\end{align*}
which completes the proof of \eqref{b7}.

Similarly, we may show \eqref{b8} by taking the high frequency limit of the data
\begin{eqnarray*}
&&{\mathbb E}\left[u^{\infty}(\hat{x},\kappa+\tau)u^{\infty}(-\hat{x},\kappa)\right]\\
&=&C_d^2(\kappa+\tau)^{\frac{d-3}2}\kappa^{\frac{d-3}2}\int_{\mathbb R^d}\int_{\mathbb R^d} e^{-{\rm i}(\kappa+\tau)\hat{x}\cdot y}e^{{\rm i}\kappa\hat{x}\cdot z}{\mathbb E}\left[f(y)f(z)\right]dydz\\
&=&C_d^2(\kappa+\tau)^{\frac{d-3}2}\kappa^{\frac{d-3}2}\int_{\mathbb R^d}\left[\int_{\mathbb R^d}K_f^r(y,z)e^{-{\rm i}\kappa\hat{x}\cdot (y-z)}dz\right]e^{-{\rm i}\tau\hat{x}\cdot y}dy\\
& =&C_d^2(\kappa+\tau)^{\frac{d-3}2}\kappa^{\frac{d-3}2}\left[\int_{\mathbb R^d}a^r(y)e^{-{\rm i}\tau\hat{x}\cdot y}dy|\kappa\hat{x}|^{-m}+\int_{\mathbb R^d}b^r(y,\kappa\hat{x})e^{-{\rm i}\tau\hat{x}\cdot y}dy\right]\\
& =&C_d^2\left(\frac{\kappa}{\kappa+\tau}\right)^{\frac{3-d}2}\kappa^{d-3-m}\widehat{a^r}(\tau\hat{x})+O(\kappa^{d-4-m}),
\end{eqnarray*}
where the residual $b^r(\cdot,\xi)\in C_0^\infty(D)$ is uniformly bounded by $|\xi|^{-m-1}$ as $|\xi|\to\infty$. 
\end{proof}

The results in Lemma \ref{lm2} imply that $a^c$ and $a^r$ can be uniquely determined by the expectation of high frequency limit of the far-field pattern. This kind of data requires the measurements at all sample paths of the random source. Next we show that $a^c$ and $a^r$ can also be uniquely determined by the far-field pattern averaged over the frequency band at a single sample path almost surely. 

The following results present some a priori estimates of the far-field pattern, which are used to show the analogue of ergodicity in the frequency domain. 

\begin{lemma}\label{lem2}
Let $f$ satisfy Assumption \ref{as1} with $m\in (d-4, d]$. For any $\hat{x}\in\mathbb S^{d-1}$, $\kappa_1,\kappa_2\geq 1$ and any fixed $N\in\mathbb N$, the following estimates hold: 
\begin{eqnarray}\label{b16}
\left|{\mathbb E}\left[u^{\infty}(\hat{x},\kappa_1)\overline{u^{\infty}(\hat{x},\kappa_2)}\right]\right|&\lesssim&\kappa_1^{\frac{d-3}{2}}\kappa_2^{\frac{d-3}{2}-m}(1+|\kappa_1-\kappa_2|)^{-N},\\\label{b17}
\left|{\mathbb E}\left[u^{\infty}(\hat{x},\kappa_1)u^{\infty}(-\hat{x},\kappa_2)\right]\right|&\lesssim&\kappa_1^{\frac{d-3}{2}}\kappa_2^{\frac{d-3}{2}-m}(1+|\kappa_1-\kappa_2|)^{-N},\\\label{b18}
\left|{\mathbb E}\left[u^{\infty}(\hat{x},\kappa_1)\overline{u^{\infty}(-\hat{x},\kappa_2)}\right]\right|&\lesssim&\kappa_1^{\frac{d-3}{2}}\kappa_2^{\frac{d-3}{2}-m}(1+\kappa_1+\kappa_2)^{-N},\\\label{b19}
\left|{\mathbb E}\left[u^{\infty}(\hat{x},\kappa_1)u^{\infty}(\hat{x},\kappa_2)\right]\right|&\lesssim&\kappa_1^{\frac{d-3}{2}}\kappa_2^{\frac{d-3}{2}-m}(1+\kappa_1+\kappa_2)^{-N}.
\end{eqnarray}
\end{lemma}
\begin{proof}
It follows from \eqref{a14} and (\ref{b6}) that 
\begin{eqnarray}\nonumber
&&\mathbb E\left[u^{\infty}(\hat{x},\kappa_1)\overline{u^{\infty}(\hat{x},\kappa_2)}\right]\\\nonumber
&=&|C_d|^2(\kappa_1\kappa_2)^{\frac{d-3}2}\int_{\mathbb R^d}\int_{\mathbb R^d} e^{-{\rm i}\kappa_1\hat{x}\cdot y}e^{{\rm i}\kappa_2\hat{x}\cdot z}{\mathbb E}\left[f(y)\overline{f(z)}\right]dydz\\\nonumber
&=&|C_d|^2(\kappa_1\kappa_2)^{\frac{d-3}2}\int_{\mathbb R^d}\left[\int_{\mathbb R^d}e^{-{\rm i}\kappa_2\hat{x}\cdot (y-z)}K^c_f(y,z)dz\right]e^{{\rm i}(\kappa_2-\kappa_1)\hat{x}\cdot y}dy\\\label{x1}
&=&|C_d|^2(\kappa_1\kappa_2)^{\frac{d-3}2}\int_{D} \sigma^c(y,\kappa_2\hat{x})e^{{\rm i}(\kappa_2-\kappa_1)\hat{x}\cdot y}dy,
\end{eqnarray}
where the symbol $\sigma^c\in\mathcal{S}^{-m}(\mathbb R^d\times\mathbb R^d)$ satisfies
\begin{eqnarray}\label{b21}
\left|\partial^{\alpha}_y\sigma^c(y,\kappa_2\hat{x})\right|\lesssim(1+\kappa_2)^{-m}
\end{eqnarray}
for any multiple index $\alpha$, and we used the fact that $\sigma^c(\cdot,\kappa_2\hat x)$ is compactly supported in $D$.

If $|\kappa_1-\kappa_2|<1$,  we have from \eqref{x1}--\eqref{b21} that 
\begin{eqnarray*}
\left|{\mathbb E}\left[u^{\infty}(\hat{x},\kappa_1)\overline{u^{\infty}(\hat{x},\kappa_2)}\right]\right|&\lesssim& (\kappa_1\kappa_2)^{\frac{d-3}2}\int_{D} |\sigma^c(y,\kappa_2\hat{x})|dy\\
&\lesssim& (\kappa_1\kappa_2)^{\frac{d-3}2}(1+\kappa_2)^{-m}\\
&\lesssim&  (\kappa_1\kappa_2)^{\frac{d-3}2}\left(\frac{2}{1+|\kappa_1-\kappa_2|}\right)^N(1+\kappa_2)^{-m}\\
&\lesssim&  2^N\kappa_1^{\frac{d-3}{2}}\kappa_2^{\frac{d-3}{2}-m}(1+|\kappa_1-\kappa_2|)^{-N}.
\end{eqnarray*}

If $|\kappa_1-\kappa_2|\geq 1$, applying the integration by parts to (\ref{x1}) with respect to $y_1$ gives 
\begin{eqnarray*}
&&\mathbb E\left[u^{\infty}(\hat{x},\kappa_1)\overline{u^{\infty}(\hat{x},\kappa_2)}\right]\\
&=&|C_d|^2(\kappa_1\kappa_2)^{\frac{d-3}2}\frac{-1}{{\rm i}(\kappa_2-\kappa_1)\hat{x}_1}\int_D\partial_{y_1}\sigma^c(y,\kappa_2\hat{x})e^{{\rm i}(\kappa_2-\kappa_1)\hat{x}\cdot y}dy\\
&=&|C_d|^2(\kappa_1\kappa_2)^{\frac{d-3}2}\left(\frac{-1}{{\rm i}(\kappa_2-\kappa_1)\hat{x}_1}\right)^N\int_D\partial^N_{y_1}c_f(y,\kappa_2\hat{x})e^{{\rm i}(\kappa_2-\kappa_1)\hat{x}\cdot y}dy.
\end{eqnarray*}
Hence 
\begin{eqnarray*}
\left|{\mathbb E}\left[u^{\infty}(\hat{x},\kappa_1)\overline{u^{\infty}(\hat{x},\kappa_2)}\right]\right|&\lesssim& (\kappa_1\kappa_2)^{\frac{d-3}2}\frac{1}{|\kappa_1-\kappa_2|^N}(1+\kappa_2)^{-m}\\
&\lesssim& \kappa_1^{\frac{d-3}{2}}\kappa_2^{\frac{d-3}{2}-m}\left(1+\frac{1}{|\kappa_1-\kappa_2|}\right)^N(1+|\kappa_1-\kappa_2|)^{-N}\\
&\lesssim& 2^N \kappa_1^{\frac{d-3}{2}}\kappa_2^{\frac{d-3}{2}-m}(1+|\kappa_1-\kappa_2|)^{-N},
\end{eqnarray*}
which concludes \eqref{b16}. 

The estimate \eqref{b17} can be obtained similarly by noting 
\begin{eqnarray}\nonumber
&&\mathbb E\left[u^{\infty}(\hat{x},\kappa_1)u^{\infty}(-\hat{x},\kappa_2)\right]\\
&=&C_d^2(\kappa_1\kappa_2)^{\frac{d-3}2}\int_{\mathbb R^d}\int_{\mathbb R^d} e^{-{\rm i}\kappa_1\hat{x}\cdot y}e^{{\rm i}\kappa_2\hat{x}\cdot z}{\mathbb E}\left[f(y)f(z)\right]dydz\nonumber\\\nonumber
&=&C_d^2(\kappa_1\kappa_2)^{\frac{d-3}2}\int_{\mathbb R^d}\left[\int_{\mathbb R^d}e^{-{\rm i}\kappa_2\hat{x}\cdot (y-z)}K^r_f(y,z)dz\right]e^{{\rm i}(\kappa_2-\kappa_1)\hat{x}\cdot y}dy\\\label{x2}
&=&C_d^2(\kappa_1\kappa_2)^{\frac{d-3}2}\int_{\mathbb R^d}\sigma^r(y,\kappa_2\hat{x})e^{{\rm i}(\kappa_2-\kappa_1)\hat{x}\cdot y}dy,
\end{eqnarray}
where the estimate is similar to \eqref{x1} with $\sigma^c$ being replaced by $\sigma^r$. 

For \eqref{b18} and \eqref{b19}, we rewrite the correlations as follows
\begin{eqnarray*}
{\mathbb E}\left[u^{\infty}(\hat{x},\kappa_1)\overline{u^{\infty}(-\hat{x},\kappa_2)}\right]&=&|C_d|^2(\kappa_1\kappa_2)^{\frac{d-3}2}\int_{\mathbb R^d}\int_{\mathbb R^d} e^{-{\rm i}\kappa_1\hat{x}\cdot y}e^{-{\rm i}\kappa_2\hat{x}\cdot z}{\mathbb E}\left[f(y)\overline{f(z)}\right]dydz,\\
{\mathbb E}\left[u^{\infty}(\hat{x},\kappa_1)u^{\infty}(\hat{x},\kappa_2)\right]&=&C_d^2(\kappa_1\kappa_2)^{\frac{d-3}2}\int_{\mathbb R^d}\int_{\mathbb R^d} e^{-{\rm i}\kappa_1\hat{x}\cdot y}e^{-{\rm i}\kappa_2\hat{x}\cdot z}{\mathbb E}\left[f(y)f(z)\right]dydz.
\end{eqnarray*}
Comparing the above formulas with \eqref{x1} and \eqref{x2}, it is easily seen that they can be estimated similarly to \eqref{b16} and \eqref{b17} by replacing $\kappa_2$ by $-\kappa_2$, respectively, which completes the proofs of \eqref{b18} and \eqref{b19}.
\end{proof}

\begin{theorem}\label{thm3}
Let $f$ satisfy Assumption \ref{as1} with $m\in (d-4,d]$. Then for all $\hat{x}\in {\mathbb S}^{d-1}$ and $\tau\geq 0$, it holds almost surely that 
\begin{eqnarray}\label{b12}
\lim_{Q\to\infty}\frac{1}{Q}\int_Q^{2Q}\kappa^{m+3-d}u^{\infty}(\hat{x},\kappa+\tau)\overline{u^{\infty}(\hat{x},\kappa)}d\kappa&=&|C_d|^2\widehat{a^c}(\tau\hat{x}),\\\label{b13}
\lim_{Q\to\infty}\frac{1}{Q}\int_Q^{2Q}\kappa^{m+3-d}u^{\infty}(\hat{x},\kappa+\tau)u^{\infty}(-\hat{x},\kappa)d\kappa&=&C_d^2\widehat{a^r}(\tau\hat{x}).
\end{eqnarray}
Moreover, $a^c$ and $a^r$ can be uniquely determined by (\ref{b12}) and (\ref{b13}), respectively, with $(\tau, \hat{x})\in\Theta$ and $\Theta\subset\mathbb R_+\times\mathbb S^{d-1}$ being any open domain.
\end{theorem}

\begin{proof}
We only give the proof of \eqref{b12} since the proof of \eqref{b13} can be obtained similarly by using \eqref{b8} in Lemma \ref{lm2}.

Based on the proof of Lemma \ref{lm2}, we multiply both sides of \eqref{eq:lm2} by $\kappa^{m+3-d}$,  take integral with respect to $\kappa$, and get
\begin{eqnarray*}
&&\frac1Q\int_Q^{2Q}\kappa^{m+3-d}{\mathbb E}\left[u^{\infty}(\hat{x},\kappa+\tau)\overline{u^{\infty}(\hat{x},\kappa)}\right]d\kappa\\
&=&|C_d|^2\widehat{a^c}(\tau\hat x)\left[\frac1Q\int_Q^{2Q}\left(\frac{\kappa}{\kappa+\tau}\right)^{\frac{3-d}2}d\kappa\right]+O\left(Q^{-1}\right).
\end{eqnarray*}
Noting 
\[
\frac1Q\int_Q^{2Q}\left(\frac{\kappa}{\kappa+\tau}\right)^{\frac{3-d}2}d\kappa\le1
\]
and
\[
\lim_{Q\to\infty}\frac1Q\int_Q^{2Q}\left(\frac{\kappa}{\kappa+\tau}\right)^{\frac{3-d}2}d\kappa\ge\lim_{Q\to\infty}\frac1Q\int_Q^{2Q}\left(\frac{Q}{Q+\tau}\right)^{\frac{3-d}2}d\kappa=1
\]
leads to 
\begin{eqnarray}\label{b23}
\lim_{Q\to\infty}\frac1Q\int_Q^{2Q}\kappa^{m+3-d}\mathbb E\left[u^{\infty}(\hat{x},\kappa+\tau)\overline{u^{\infty}(\hat{x},\kappa)}\right]d\kappa=|C_d|^2\widehat{a^c}(\tau\hat x).
\end{eqnarray}

To characterize the error between \eqref{b12} and \eqref{b23}, we define an auxiliary process
\begin{eqnarray*}\label{b24}
Y(\hat{x},\kappa):=\kappa^{m+3-d}\left(u^{\infty}(\hat{x},\kappa+\tau)\overline{u^{\infty}(\hat{x},\kappa)}-{\mathbb E}\left[u^{\infty}(\hat{x},\kappa+\tau)\overline{u^{\infty}(\hat{x},\kappa)}\right]\right).
\end{eqnarray*}
For convenience, we denote by
\begin{eqnarray*}\label{b26}
U(\hat{x},\kappa):=\frac{1}{2}\left[u^{\infty}(\hat{x},\kappa)+\overline{u^{\infty}(\hat{x},\kappa)}\right],\quad V(\hat{x},\kappa):=\frac{1}{2{\rm i}}\left[u^{\infty}(\hat{x},\kappa)-\overline{u^{\infty}(\hat{x},\kappa)}\right]
\end{eqnarray*}
the real and imaginary parts of $u^{\infty}(\hat{x},\kappa)$, respectively. Then $u^{\infty}(\hat{x},\kappa+\tau)\overline{u^{\infty}(\hat{x},\kappa)}$ can be rewritten as
\begin{eqnarray*}\nonumber
u^{\infty}(\hat{x},\kappa+\tau)\overline{u^{\infty}(\hat{x},\kappa)}&=&\left[U(\hat{x},\kappa+\tau)+{\rm i}V(\hat{x},\kappa+\tau)\right]\left[U(\hat{x},\kappa)-{\rm i}V(\hat{x},\kappa)\right]\\\nonumber
&=&\frac{1+{\rm i}}2\left[U^2(\hat{x},\kappa)+U^2(\hat{x},\kappa+\tau)+V^2(\hat{x},\kappa)+V^2(\hat{x},\kappa+\tau)\right]\\\nonumber
&&-\frac12\left(U(\hat{x},\kappa)-U(\hat{x},\kappa+\tau)\right)^2-\frac12\left(V(\hat{x},\kappa)-V(\hat{x},\kappa+\tau)\right)^2\\\label{b27}
&&-\frac{\rm i}2\left(U(\hat{x},\kappa+\tau)+V(\hat{x},\kappa)\right)^2-\frac{\rm i}2\left(V(\hat{x},\kappa+\tau)-U(\hat{x},\kappa)\right)^2.
\end{eqnarray*}
Define $\Gamma=\Gamma_1\cup \Gamma_2$, where 
\begin{eqnarray*}
&&\Gamma_1:=\{U(\hat{x},\kappa),V(\hat{x},\kappa),U(\hat{x},\kappa+\tau),V(\hat{x},\kappa+\tau)\},\\
&&\Gamma_2:=\{U(\hat{x},\kappa)-U(\hat{x},\kappa+\tau),V(\hat{x},\kappa)-V(\hat{x},\kappa+\tau),\\
&&\qquad\quad U(\hat{x},\kappa+\tau)+V(\hat{x},\kappa),V(\hat{x},\kappa+\tau)-U(\hat{x},\kappa)\},
\end{eqnarray*}
and let $W_{\kappa}$ be any random field in $\Gamma$. Based on these notations, we have
\begin{eqnarray*}\label{b28}
Y(\hat{x},\kappa)=\sum_{W_{\kappa}\in\Gamma}C(W_{\kappa})\kappa^{m+3-d}(W_{\kappa}^2-{\mathbb E}W_{\kappa}^2),
\end{eqnarray*}
where $C(W_{\kappa})\in\{\frac{1+{\rm i}}2, -\frac12,-\frac{\rm i}2\}$ is a constant depending on $W_{\kappa}$. 

Now it suffices to show for all $W_\kappa\in\Gamma$ that 
\begin{eqnarray}\label{b29}
\lim_{Q\to\infty}\frac{1}{Q}\int_Q^{2Q}\kappa^{m+3-d}(W_{\kappa}^2-{\mathbb E}W_{\kappa}^2)d\kappa=0.
\end{eqnarray}
 Hence
\begin{eqnarray*}\label{b25}
\lim_{Q\to\infty}\frac{1}{Q}\int_Q^{2Q}Y(\hat{x},\kappa)d\kappa=0,
\end{eqnarray*}
which, together with \eqref{b23}, yields \eqref{b12}.

To prove \eqref{b29}, by denoting the $\mathbb R$-valued centered random field
\[
X_\kappa:=\kappa^{m+3-d}(W_{\kappa}^2-{\mathbb E}W_{\kappa}^2)
\]
according to \cite[Theorem 4.1]{CHL} and \cite[Lemma 6]{LLW2}, one only need to show that there exist some constants $\eta\ge0$, $\beta>0$ and $C_{\tau}>0$ independent of $\kappa$ and $t$ such that
\[
|\mathbb E[X_{\kappa}X_{\kappa+t}]|\le C_\tau(1+|t-\eta|)^{-\beta}\quad\forall~\kappa,t>0.
\]
More precisely, it suffices to show
\begin{eqnarray}\label{eq:thm3}
&&\left|\mathbb E\left[\kappa^{m+3-d}(W_{\kappa}^2-{\mathbb E}W_{\kappa}^2)(\kappa+t)^{m+3-d}(W_{\kappa+t}^2-{\mathbb E}W_{\kappa+t}^2)\right]\right|\notag\\
&=&2\left({\mathbb E}\left[\kappa^{\frac{m+3-d}{2}}(\kappa+t)^{\frac{m+3-d}{2}}W_{\kappa}W_{\kappa+t}\right]\right)^2\notag\\
&\le&C_\tau(1+|t-\eta|)^{-\beta}
\end{eqnarray}
for all $W_\kappa\in\Gamma$, where in the first step we used \cite[Lemma 4.2]{CHL} and the fact that $W_\kappa$ is Gaussian.

For any $W_{\kappa}\in\Gamma_1$, based on the identities 
\begin{eqnarray*}\label{b31}
&&U(\hat{x},\kappa_1)U(\hat{x},\kappa_2)=\frac{1}{4}\left[u^{\infty}(\hat{x},\kappa_1)+\overline{u^{\infty}(\hat{x},\kappa_1)}\right]\left[u^{\infty}(\hat{x},\kappa_2)+\overline{u^{\infty}(\hat{x},\kappa_2)}\right],\\\label{b32}
&&V(\hat{x},\kappa_1)V(\hat{x},\kappa_2)=-\frac{1}{4}\left[u^{\infty}(\hat{x},\kappa_1)-\overline{u^{\infty}(\hat{x},\kappa_1)}\right]\left[u^{\infty}(\hat{x},\kappa_2)-\overline{u^{\infty}(\hat{x},\kappa_2)}\right]
\end{eqnarray*}
and Lemma \ref{lem2}, we get 
\begin{eqnarray*}\label{b34}
&&\left|{\mathbb E}\left[U(\hat{x},\kappa_1)U(\hat{x},\kappa_2)\right]\right|\lesssim\kappa_1^{\frac{d-3}{2}}\kappa_2^{\frac{d-3}{2}-m}(1+|\kappa_1-\kappa_2|)^{-N},\\\label{b35}
&&\left|{\mathbb E}\left[V(\hat{x},\kappa_1)V(\hat{x},\kappa_2)\right]\right|\lesssim\kappa_1^{\frac{d-3}{2}}\kappa_2^{\frac{d-3}{2}-m}(1+|\kappa_1-\kappa_2|)^{-N}.
\end{eqnarray*}
As a result, it holds
\begin{eqnarray}\nonumber
&&\left|{\mathbb E}\left[\kappa^{\frac{m+3-d}{2}}(\kappa+t)^{\frac{m+3-d}{2}}W_\kappa W_{\kappa+t}\right]\right|\\\nonumber
&\lesssim& \kappa^{\frac{m+3-d}{2}}(\kappa+t)^{\frac{m+3-d}{2}}\kappa^{\frac{d-3}{2}}(\kappa+t)^{\frac{d-3}{2}-m}(1+t)^{-N}\\\nonumber
&\lesssim&\left(\frac{\kappa}{\kappa+t}\right)^{\frac{m}{2}}(1+t)^{-N}\\\label{b37}
&\lesssim& (1+t)^{-\left(N+\frac{m}2\wedge0\right)},
\end{eqnarray}
where we used the facts that
\[
\left(\frac{\kappa}{\kappa+t}\right)^{\frac{m}{2}}\le1
\]
for $m\ge0$ and
\[
\left(\frac{\kappa}{\kappa+t}\right)^{\frac{m}{2}}=\left(1+\frac{t}{\kappa}\right)^{-\frac m2}\le(1+t)^{-\frac m2}
\]
for $m<0$.
Then \eqref{b37} implies that (\ref{eq:thm3}) holds for $W_{\kappa}\in\Gamma_1$ by choosing $N>-\frac m2$.

For any $W_{\kappa}\in\Gamma_2$, we take $W_{\kappa}=U(\hat{x},\kappa)-U(\hat{x},\kappa+\tau)$ for instance. The other cases can be estimated similarly. Note that 
\begin{eqnarray*}
\left|{\mathbb E}\left[W_{\kappa}W_{\kappa+t}\right]\right|&=&\left|{\mathbb E}\left[(U(\hat{x},\kappa)-U(\hat{x},\kappa+\tau))(U(\hat{x},\kappa+t)-U(\hat{x},\kappa+\tau+t))\right]\right|\\
&\leq& \left|{\mathbb E}\left[U(\hat{x},\kappa)U(\hat{x},\kappa+t)\right]\right|+\left|{\mathbb E}\left[U(\hat{x},\kappa)U(\hat{x},\kappa+\tau+t)\right]\right|\\
&&+\left|{\mathbb E}\left[U(\hat{x},\kappa+\tau)U(\hat{x},\kappa+t)\right]\right|+\left|{\mathbb E}\left[U(\hat{x},\kappa+\tau)U(\hat{x},\kappa+\tau+t)\right]\right|\\
&\lesssim& \kappa^{\frac{d-3}{2}}(\kappa+t)^{\frac{d-3}2-m}(1+t)^{-N}+\kappa^{\frac{d-3}{2}}(\kappa+\tau+t)^{\frac{d-3}2-m}(1+t+\tau)^{-N}\\
&&+(\kappa+\tau)^{\frac{d-3}{2}}(\kappa+t)^{\frac{d-3}2-m}(1+|t-\tau|)^{-N}\\
&&+(\kappa+\tau)^{\frac{d-3}{2}}(\kappa+\tau+t)^{\frac{d-3}2-m}(1+t)^{-N}\\
&\lesssim& \kappa^{\frac{d-3}{2}}(\kappa+t)^{\frac{d-3}2-m}(1+t)^{-N}+(\kappa+\tau)^{\frac{d-3}{2}}(\kappa+t)^{\frac{d-3}2-m}(1+|t-\tau|)^{-N}\\
&&+(\kappa+\tau)^{\frac{d-3}{2}}(\kappa+\tau+t)^{\frac{d-3}2-m}(1+t)^{-N}.
\end{eqnarray*}
It then leads to 
\begin{eqnarray}\nonumber
&&\left|{\mathbb E}\left[\kappa^{\frac{m+3-d}{2}}(\kappa+t)^{\frac{m+3-d}{2}}W_{\kappa}W_{\kappa+t}\right]\right|\\\nonumber
&\lesssim& \kappa^{\frac{m}{2}}(\kappa+t)^{-\frac{m}{2}}(1+t)^{-N}+\kappa^{\frac{m}{2}}\left(\frac{\kappa}{\kappa+\tau}\right)^{\frac{3-d}{2}}(\kappa+t)^{-\frac{m}{2}}(1+|t-\tau|)^{-N}\\\nonumber
&&+\left(\frac{\kappa}{\kappa+\tau}\right)^{\frac{3-d}{2}}\left(\frac{\kappa}{\kappa+t+\tau}\right)^{\frac m2}\left(\frac{\kappa+t}{\kappa+t+\tau}\right)^{\frac{m+3-d}2}(1+t)^{-N}\\\label{b39}
&\lesssim& (1+t)^{-(N+m\wedge0)}+(1+|t-\tau|)^{-N},
\end{eqnarray}
where we used the estimates in \eqref{b37} and the facts that
\[
\left(\frac{\kappa}{\kappa+t+\tau}\right)^{\frac m2}\left(\frac{\kappa+t}{\kappa+t+\tau}\right)^{\frac{m+3-d}2}\le\left(\frac{\kappa}{\kappa+t+\tau}\right)^{\frac m2}\left(\frac{\kappa+t}{\kappa+t+\tau}\right)^{\frac{m}2}\le1
\]
for $m\ge0$ and 
\begin{eqnarray*}
\left(\frac{\kappa}{\kappa+t+\tau}\right)^{\frac m2}\left(\frac{\kappa+t}{\kappa+t+\tau}\right)^{\frac{m+3-d}2}&\le&\left(\frac{\kappa}{\kappa+t+\tau}\right)^m
\lesssim(1+t+\tau)^{-m}\\
&\lesssim&(1+\tau)^{-m}(1+t)^{-m}
\end{eqnarray*}
for $m<0$. It then completes the proof of \eqref{eq:thm3} for $W_\kappa\in\Gamma_2$. 

Combining \eqref{b37} and \eqref{b39} yields \eqref{eq:thm3} for all $W_\kappa\in\Gamma$ and deduces \eqref{b12}.

Since  $a^c$ and $a^r$ are analytic, they can be uniquely determined by $\{\widehat{a^c}(\tau\hat{x})\}$ and $\{\widehat{a^r}(\tau\hat{x})\}$, where $(\tau,\hat{x})\in \Theta$ with $\Theta$ being any open subdomain of $\mathbb R_+\times\mathbb S^{1}$. 
\end{proof}

\section{Biharmonic waves}

In this section, we study the direct and inverse source problems for the stochastic biharmonic wave equation 
\begin{eqnarray}\label{c1}
\Delta^2u-\kappa^4u=f\quad{\rm in}~\mathbb R^d,
\end{eqnarray}
where $f$ is assumed to be a $\mathbb C$-valued GMIG random field satisfying Assumption \ref{as1} with $m\in (d-6,d]$. In addition, the wave field $u$ and its Laplacian $\Delta u$ are required to satisfy the Sommerfeld radiation condition 
\begin{eqnarray}\label{c2}
\lim_{|x|\to\infty}|x|^{\frac{d-1}{2}}(\partial_{|x|}u-{\rm i}\kappa u)=0,\quad \lim_{|x|\to\infty}|x|^{\frac{d-1}{2}}(\partial_{|x|}\Delta u-{\rm i}\kappa\Delta u)=0.
\end{eqnarray}

Given $f$, the well-posedness of the problem \eqref{c1}--\eqref{c2} was studied in \cite[Theorem 3.2]{LW4} and is given below. 

\begin{theorem}
Let $f$ satisfy Assumption \ref{as1} with $m\in(d-6,d]$. Then the problem (\ref{c1})--(\ref{c2}) admits a unique solution 
\begin{eqnarray}\label{c3}
u(x,\kappa)=-\int_{\mathbb R^d}F_d(x,y,\kappa)f(y)dy
\end{eqnarray}
in the sense of distributions such that $u\in W_{\rm loc}^{\gamma,q}(\mathbb R^d)$ almost surely for any $q>1$ and $0<\gamma<\min\left\{\frac{6-d+m}{2},\frac{6-d+m}{2}+\left(\frac{1}{q}-\frac{1}{2}\right)d\right\}$.
\end{theorem}

Next we address the inverse problem for biharmonic waves. Combining  \eqref{a29} and \eqref{c3} leads to 
\[
u(x,\kappa)=\frac{e^{{\rm i}\kappa|x|}}{|x|^{\frac{d-1}2}}\left(u^\infty(\hat x,\kappa)+O(|x|^{-1})\right),
\]
where the far-field pattern is given by 
\begin{eqnarray}\label{c5}
u^{\infty}(\hat{x},\kappa)=-\frac{C_d}2\kappa^{\frac{d-7}2}\int_{\mathbb R^d}e^{-{\rm i}\kappa\hat{x}\cdot y}f(y)dy.
\end{eqnarray}

It is easy to note from \eqref{b6} and \eqref{c5} that the procedure used in Section \ref{sec:acoustic} for acoustic waves is applicable for biharmonic waves. The following is the main result for the inverse source problem of the biharmonic wave equation. 

\begin{theorem}
Let $f$ satisfy Assumption \ref{as1} with $m\in (d-6,d]$. Then for all $\hat{x}\in {\mathbb S}^{d-1}$ and $\tau\geq 0$, it holds almost surely that 
\begin{eqnarray}\label{c7}
&&\lim_{Q\to\infty}\frac{1}{Q}\int_Q^{2Q}\kappa^{m+7-d}u^{\infty}(\hat{x},\kappa+\tau)\overline{u^{\infty}(\hat{x},\kappa)}d\kappa=\frac14|C_d|^2\widehat{a^c}(\tau\hat{x}),\\\label{c8}
&&\lim_{Q\to\infty}\frac{1}{Q}\int_Q^{2Q}\kappa^{m+7-d}u^{\infty}(\hat{x},\kappa+\tau)u^{\infty}(-\hat{x},\kappa)d\kappa=\frac14C_d^2\widehat{a^r}(\tau\hat{x}).
\end{eqnarray}
Moreover, $a^c$ and $a^r$ can be uniquely determined by (\ref{c7}) and (\ref{c8}), respectively, with $(\tau, \hat{x})\in\Theta$ and $\Theta\subset\mathbb R_+\times\mathbb S^{d-1}$ being any open domain.
\end{theorem}

\begin{proof}
A simple calculation yields 
\begin{eqnarray*}
&&{\mathbb E}\left[u^{\infty}(\hat{x},\kappa+\tau)\overline{u^{\infty}(\hat{x},\kappa)}\right]\notag\\
&=&\frac14|C_d|^2(\kappa+\tau)^{\frac{d-7}2}\kappa^{\frac{d-7}2}\int_{\mathbb R^d}\int_{\mathbb R^d} e^{-{\rm i}(\kappa+\tau)\hat{x}\cdot y}e^{{\rm i}\kappa\hat{x}\cdot z}{\mathbb E}\left[f(y)\overline{f(z)}\right]dydz\notag\\
&=&\frac14|C_d|^2(\kappa+\tau)^{\frac{d-7}2}\kappa^{\frac{d-7}2}\int_{\mathbb R^d}\left[\int_{\mathbb R^d}K_f^c(y,z)e^{-{\rm i}\kappa\hat{x}\cdot (y-z)}dz\right]e^{-{\rm i}\tau\hat{x}\cdot y}dy\notag\\
& =&\frac14|C_d|^2(\kappa+\tau)^{\frac{d-7}2}\kappa^{\frac{d-7}2}\left[\int_{\mathbb R^d}a^c(y)e^{-{\rm i}\tau\hat{x}\cdot y}dy|\kappa\hat{x}|^{-m}+\int_{\mathbb R^d}b^c(y,\kappa\hat{x})e^{-{\rm i}\tau\hat{x}\cdot y}dy\right]\notag\\
& =&\frac14|C_d|^2\left(\frac{\kappa}{\kappa+\tau}\right)^{\frac{7-d}2}\kappa^{d-7-m}\widehat{a^c}(\tau\hat{x})+O(\kappa^{d-8-m}),
\end{eqnarray*}
which gives 
\[
\lim_{\kappa\to\infty}\kappa^{m+7-d}\mathbb E\left[u^{\infty}(\hat{x},\kappa+\tau)\overline{u^{\infty}(\hat{x},\kappa)}\right]=\frac14|C_d|^2\widehat{a^c}(\tau\hat{x}).
\]

Using 
\[
1\ge\lim_{Q\to\infty}\frac1Q\int_Q^{2Q}\left(\frac{\kappa}{\kappa+\tau}\right)^{\frac{7-d}2}d\kappa\ge\lim_{Q\to\infty}\frac1Q\int_Q^{2Q}\left(\frac{Q}{Q+\tau}\right)^{\frac{7-d}2}d\kappa=1
\]
and the estimates of the correlations of the far-field pattern $u^\infty$ at different frequencies, which can be obtained by following the same procedure as the one used in Lemma \ref{lem2}, we may replace the high frequency limit in the above result by the limit of the averaged data over the frequency band at a single sample path with probability one, i.e., 
\[
\lim_{Q\to\infty}\frac1Q\int_Q^{2Q}\kappa^{m+7-d}u^{\infty}(\hat{x},\kappa+\tau)\overline{u^{\infty}(\hat{x},\kappa)}d\kappa=\frac14|C_d|^2\widehat{a^c}(\tau\hat{x}). 
\]
The recovery formula \eqref{c8} can be obtained similarly and the details are omitted for brevity. 
\end{proof}

\section{Electromagnetic waves}

This section is concerned with the direct and inverse source problems for electromagnetic waves. The inverse random source problem for Maxwell's equations was considered in \cite{LW2}, where the source was assumed to be a centered GMIG random vector field whose real and imaginary parts were independent and identically distributed. Under this assumption, the relation operator of the random source vanishes, and the random source is only determined by its covariance operator. The strength matrix of the covariance operator was proved to be uniquely determined by the phased near-field data of the electric field. 

In this work, we remove the assumption that the real and imaginary parts of the random source are independent and identically distributed, and investigate the recovery of the strengths of both the covariance and relation operators for the random source from the far-field pattern of the electric field.

Consider the stochastic Maxwell's equations 
\begin{eqnarray}\label{d1}
\nabla\times {\bm E}-{\rm i}\kappa {\bm H}=0,\qquad \nabla\times{\bm H}+{\rm i}\kappa {\bm E}={\bm f}\quad{\rm in}\;\;{\mathbb R^3},
\end{eqnarray}
where ${\bm E}$ and ${\bm H}$ are the electric and magnetic fields, respectively, and the random source ${\bm f}$ represents the electric current density satisfying the following assumption with $d=3$.

\begin{assumption}\label{as2}
The electric current density ${\bm f}$ is assumed to be a $\mathbb C^d$-valued centered GMIG random
field of order $-m$ in a bounded domain $D\subset\mathbb R^d$. The principal symbols of its covariance operator $\mathcal C_{\bm f}$ and relation operator $\mathcal R_{\bm f}$ have the forms 
$A^c(x)|\xi|^{-m}$ and $A^r(x)|\xi|^{-m}$, respectively, where $A^c, A^r\in C_0^{\infty}(D;{\mathbb C}^{d\times d})$.
\end{assumption}

As usual, an appropriate radiation condition is required for \eqref{d1}. Note that ${\bm f}\in\bm{\mathcal{D}}'(\mathbb R^3;\mathbb C^3)$ is a distribution, and hence (\ref{d1}) is interpreted in the sense of distributions. In \cite{LW2}, 
the following weak Silver--M\"{u}ller radiation condition was proposed to the electromagnetic fields: 
\begin{eqnarray}\label{d2}
\lim_{r\to\infty}\int_{|x|=r}\left({\bm H}\times\hat{x}-{\bm E}\right)\cdot{\bm \phi}ds=0\qquad\forall~{\bm \phi}\in\bm{\mathcal{D}}(\mathbb R^3). 
\end{eqnarray}

In addition to Assumption \ref{as2} with $m\in(2,3]$, $\bm f$ is required to be a distribution belonging to the space
\begin{eqnarray*}\label{d3}
{\mathbb X}:=\left\{{\bm U}\in \bm{\mathcal{D}}'(\mathbb R^3): \int_{\mathbb R^3}{\bm U}\cdot\left(\nabla(\nabla\cdot{\bm\phi})\right)dx=0\quad\forall~{\bm \phi}\in\bm{\mathcal{D}}(\mathbb R^3)\right\}.
\end{eqnarray*}
Apparently, the space $\mathbb X$ is non-empty: if $\bm f$ is smooth enough and divergence-free, then $\bm f\in\mathbb X$. In fact, $\mathbb X$ can be regarded as the space of all distributions which are divergence-free in the sense of distributions. 
The weak divergence-free condition ensures that Maxwell's equations \eqref{d1} can be reduced to the Helmholtz equation and the electric field has an integral representation in terms of the source, where the integral kernel is exactly the fundamental solution to the Helmholtz equation. The details can be found in \cite{LW2}. 

The following result concerns the well-posedness of \eqref{d1}--\eqref{d2} with a relaxed assumption on the order $m$ of the random source.

\begin{theorem}
Let $\bm f\in\mathbb X$ satisfy Assumption \ref{as2} with $m\in(-1,3]$.
The problem (\ref{d1})--(\ref{d2}) admits a unique solution $(\bm E,\bm H)$ with $\bm E\in\mathbb X\cap\bm W^{\gamma,q}_{loc}(\mathbb R^3)$ and $\bm H\in(\bm W^{-\gamma,p}(curl))'$ almost surely for $q>1$, $0<\gamma<\min\left\{\frac{1+m}2,\frac m2+\frac3q-1\right\}$ and $p$ satisfying $\frac1p+\frac1q=1$. Moreover, the electric field has the form
\begin{eqnarray*}\label{d4}
{\bm E}(x,\kappa)={\rm i}\kappa\int_{\mathbb R^3}\Phi_3(x,y,\kappa){\bm f}(y)dy.
\end{eqnarray*}
\end{theorem}

The proof can be obtained directly from the well-posedness of the Helmholtz equation given in Theorem \ref{thm1} and \cite[Corollary 2.3]{LW2}. The details are omitted here.  

To recover $A^c$ and $A^r$ of the covariance and relation operators for the random source, respectively, we consider the far-field pattern of the electric field
\begin{eqnarray*}\label{d5}
{\bm E}^{\infty}(\hat{x},\kappa)={\rm i}\kappa C_3\int_{\mathbb R^3}e^{-{\rm i}\kappa\hat{x}\cdot y}{\bm f}(y)dy,
\end{eqnarray*}
which is obtained from the asymptotic behavior of the fundamental solution $\Phi_3$ given in \eqref{a25}.
By similar arguments as those for the Helmholtz equation in Theorem \ref{thm3}, we can establish the following uniqueness theorem.

\begin{theorem}
Let ${\bm f}\in\mathbb X$ satisfy Assumption \ref{as2} with $m\in \left(-1,3\right]$. Then for all $\hat{x}\in {\mathbb S}^{2}$ and $\tau\geq 0$, it holds almost surely that 
\begin{eqnarray}\label{d6}
&&\lim_{Q\to\infty}\frac{1}{Q}\int_Q^{2Q}\kappa^{m-2}\bm E^{\infty}(\hat{x},\kappa+\tau)\overline{\bm E^{\infty}(\hat{x},\kappa)}^\top d\kappa=\frac{1}{16\pi^2}\widehat{A^c}(\tau\hat{x}),\\\label{d7}
&&\lim_{Q\to\infty}\frac{1}{Q}\int_Q^{2Q}\kappa^{m-2}\bm E^{\infty}(\hat{x},\kappa+\tau)\bm E^{\infty}(-\hat{x},\kappa)^\top d\kappa=-\frac{1}{16\pi^2}\widehat{A^r}(\tau\hat{x}),
\end{eqnarray}
where the Fourier transform of a matrix $A=[a_{jl}]_{j,l=1,\cdots,d}$ is defined by
$\widehat{A}=\left[\widehat{a_{jl}}\right]_{j,l=1,2,3}.$

Moreover, the strength matrices $A^c$ and $A^r$ are uniquely determined by (\ref{d6}) and (\ref{d7}), respectively, with $(\tau, \hat{x})\in\Theta$ and $\Theta\subset\mathbb R_+\times\mathbb S^{2}$ being any open domain.
\end{theorem}

\section{Elastic waves}

This section is devoted to the direct and inverse random source problems for elastic waves. Consider the stochastic Navier equation in a homogeneous medium
\begin{eqnarray}\label{e1}
\mu\Delta{\bm u}+(\lambda+\mu)\nabla\nabla\cdot{\bm u}+\omega^2{\bm u}={\bm f}\quad{\rm in}~{\mathbb R^d},
\end{eqnarray}
where $\omega>0$ is the angular frequency, ${\bm u}\in {\mathbb C^d}$ is the displacement,  $\lambda$ and $\mu$ denote the Lam\'{e} parameters satisfying $\mu>0$ and $\lambda+2\mu>0$ such that the second order partial differential operator $\Delta^*:=\mu\Delta+(\lambda+\mu)\nabla\nabla\cdot$ is strongly elliptic (cf. \cite[section 10.4]{M}), 
and the source $\bm f$ is assumed to be a $\mathbb C^d$-valued GMIG random field satisfying Assumption \ref{as2} with some restrictions on $m$ to be given later.

By the Helmholtz decomposition \cite[Appendix B]{BLZ}, the displacement ${\bm u}$ outside the support $D$ of the random source can be decomposed as $\bm u=\bm u_p+\bm u_s$, where the compressional and shear parts $\bm u_p$ and $\bm u_s$ are defined by
\begin{equation*}
{\bm u}_p:=-\frac{1}{\kappa_p^2}\nabla\nabla\cdot {\bm u},\quad
{\bm u}_s:=\frac{1}{\kappa_s^2}\nabla\times(\nabla\times{\bm u})\quad{\rm in}~\mathbb R^d\setminus\overline{D},
\end{equation*}
where, for $d=3$, `$\nabla\times$' denotes the classical curl operator; for $d=2$, `$\nabla\times(\nabla\times\cdot)$' is defined by
\[
\nabla\times(\nabla\times\bm u):=\left(\partial_{x_1}\partial_{x_2}u_2-\partial^2_{x_2}u_1,\partial_{x_1}\partial_{x_2}u_1-\partial^2_{x_1}u_2\right)^\top
\]
for $\bm u=(u_1,u_2)^\top$. The Kupradze--Sommerfeld radiation condition requires that ${\bm u}_p$ and ${\bm u}_s$ satisfy the Sommerfeld radiation condition 
\begin{eqnarray}\label{e2}
\lim_{|x|\to\infty}|x|^{\frac{d-1}{2}}\left(\partial_{|x|}{\bm u}_p-{\rm i}\kappa_p{\bm u}_p\right)=0,\quad\lim_{|x|\to\infty}|x|^{\frac{d-1}{2}}\left(\partial_{|x|}{\bm u}_s-{\rm i}\kappa_s{\bm u}_s\right)=0.
\end{eqnarray}

The well-posedness of the problem (\ref{e1})--(\ref{e2}) was investigated in \cite{LW3}, where the random source $\bm f$ was assumed to be $\mathbb R^d$-valued satisfying Assumption \ref{as2} with $m\in(d-1,d]$. The following result gives the well-posedness of the problem (\ref{e1})--(\ref{e2}) with a $\mathbb C^d$-valued random source and a relaxed condition on the order $m$.  

\begin{theorem}
Let ${\bm f}$ satisfy Assumption \ref{as2} with $m\in (d-4,d]$. The problem (\ref{e1})--(\ref{e2}) admits a unique solution ${\bm u}\in {\bm W}_{\rm loc}^{\gamma, q}(\mathbb R^d)$ almost surely given by 
\begin{eqnarray}\label{e4}
{\bm u}(x,\omega)=-\int_{\mathbb R^d}{\bm G}_d(x,y,\omega){\bm f}(y)dy
\end{eqnarray}
for any $q>1$ and $0<\gamma<\min\left\{\frac{4-d+m}2,\frac{4-d+m}2+\left(\frac1q-\frac12\right)d\right\}$. 
\end{theorem}

Similar to the operator $\mathcal H_\kappa$ defined in Theorem \ref{thm1}, the operator 
\[
(\bm{\mathcal H}_\omega \bm f)(x):=-\int_{\mathbb R^d}\bm G_d(x,y,\omega)\bm f(y)dy
\]
generated by the Green tensor $\bm G_d$ is also bounded from $\bm H^{-s_1}(D)$ to $\bm H^{s_2}(G)$ with $s_1,s_2\ge0$ and $s_1+s_2\in(0,2]$, since the Green tensor $\bm G_d$ has the same singularity as the fundamental solution $\Phi_d$ to the Helmholtz equation. Hence, the proof of the above theorem can be obtained following the same procedure as the one used in Theorem \ref{thm1}, and is omitted here.

Based on the asymptotic behavior of the Green tensor $\bm G_d$ given in \eqref{a31}, we can rewrite $\bm u$ in \eqref{e4} as the following asymptotic expansion 
\begin{eqnarray*}
{\bm u}(x,\omega)=\frac{e^{{\rm i}\kappa_p|x|}}{|x|^{\frac{d-1}{2}}}{\bm u}^{\infty}_p(\hat{x},\omega)+\frac{e^{{\rm i}\kappa_s|x|}}{|x|^{\frac{d-1}{2}}}{\bm u}^{\infty}_s(\hat{x},\omega)+O\left(|x|^{-\frac{d+1}{2}}\right), \quad |x|\to\infty,
\end{eqnarray*}
where 
\begin{eqnarray}\label{e6}
&&{\bm u}_p^{\infty}(\hat{x},\omega)=-C_dc_p^{\frac{d+1}2}\omega^{\frac{d-3}2}\hat{x}\hat{x}^\top \int_{\mathbb R^d}e^{-{\rm i}\kappa_p\hat{x}\cdot y}{\bm f}(y)dy,\\\label{e7}
&&{\bm u}_s^{\infty}(\hat{x},\omega)=-C_dc_s^{\frac{d+1}2}\omega^{\frac{d-3}2}\left({\bm I}-\hat{x}\hat{x}^\top \right)\int_{\mathbb R^d}e^{-{\rm i}\kappa_s\hat{x}\cdot y}{\bm f}(y)dy
\end{eqnarray}
are known as the compressional and shear far-field patterns of the scattered field ${\bm u}$, respectively. 
Due to the presence of matrices $\hat{x}\hat{x}^\top $ and ${\bm I}-\hat{x}\hat{x}^\top $ in \eqref{e6}--\eqref{e7}, each component of ${\bm u}_p^{\infty}(\hat{x},\omega)$ and ${\bm u}_s^{\infty}(\hat{x},\omega)$ consists of combinations of all components of the random source ${\bm f}$, which makes it more complicated than the cases for acoustic waves, biharmonic waves, and electromagnetic waves. 

Define vectors 
\begin{eqnarray*}
{\bm v}_{p,j}:=\hat{x}_j\hat{x},\quad {\bm v}_{s,j}:={\bm e}_j-{\bm v}_{p,j},\quad j=1,\cdots,d,
\end{eqnarray*}
where $\hat{x}=(\hat{x}_1,\cdots,\hat{x}_d)^{\top}$ and $\bm e_j$ is the unit vector in $\mathbb R^d$ with its $j$th entry being one. For any matrix $A=[a_{jl}]_{j,l=1,\cdots,d}$, define a reshape operator ${\mathscr R}:\mathbb R^{d\times d}\to\mathbb R^{d^2}$ by 
\[
\mathscr R(A):=(a_{11},\cdots,a_{1d},\cdots,a_{d1},\cdots,a_{dd})^\top,
\]
which rearranges the entries of matrix $A$ in rows into a vector. 
Let ${\bm u}_p^{\infty}=(u_{p,1}^\infty,\cdots,u_{p,d}^\infty)^\top$ and ${\bm u}_s^{\infty}=(u_{s,1}^\infty,\cdots,u_{s,d}^\infty)^\top$. Then, according to \eqref{e6}--\eqref{e7}, we get the following expressions for components of $\bm u_p^\infty$ and $\bm u_s^\infty$:
\begin{eqnarray}\label{e9}
&&u_{p,j}^{\infty}(\hat{x},\omega)=-C_dc_p^{\frac{d+1}2}\omega^{\frac{d-3}2}\int_{\mathbb R^d}e^{-{\rm i}c_p\omega\hat{x}\cdot y}{\bm v}_{p,j}\cdot{\bm f}(y)dy,\\\label{e10}
&&u_{s,j}^{\infty}(\hat{x},\omega)=-C_dc_s^{\frac{d+1}2}\omega^{\frac{d-3}2}\int_{\mathbb R^d}e^{-{\rm i}c_s\omega\hat{x}\cdot y}{\bm v}_{s,j}\cdot{\bm f}(y)dy.
\end{eqnarray}

Based on the relationship between the kernel and the symbol of the covariance operator given in (\ref{a22}) as well as \eqref{e9}, we obtain 
\begin{eqnarray*}
&&{\mathbb E}\left[u_{p,j}^{\infty}\left(\hat{x},c_sc_p^{-1}(\omega+\tau)\right)\overline{u_{p,l}^{\infty}\left(\hat{x},c_sc_p^{-1}\omega\right)}\right]\\
&&=|C_d|^2c_p^4c_s^{d-3}(\omega+\tau)^{\frac{d-3}2}\omega^{\frac{d-3}2}\int_{\mathbb R^d}\int_{\mathbb R^d}e^{-{\rm i}c_s(\omega+\tau)\hat{x}\cdot y}e^{{\rm i}c_s\omega\hat{x}\cdot z}{\mathbb E}\left[{\bm v}_{p,j}\cdot {\bm f(y)} {\bm v}_{p,l}\cdot\overline{{\bm f}(z)}\right]dydz\\
&&=|C_d|^2c_p^4c_s^{d-3}(\omega+\tau)^{\frac{d-3}2}\omega^{\frac{d-3}2}\int_{\mathbb R^d}\int_{\mathbb R^d}e^{-{\rm i}c_s(\omega+\tau)\hat{x}\cdot y}e^{{\rm i}c_s\omega\hat{x}\cdot z}{\mathbb E}\left[\mathscr R ({\bm v}_{p,j}{\bm v}_{p,l}^{\top})\cdot\mathscr R ({\bm {f}(y)}\overline{{\bm f}(z)}^{\top})\right]dydz\\
&&=|C_d|^2c_p^4c_s^{d-3}(\omega+\tau)^{\frac{d-3}2}\omega^{\frac{d-3}2}\int_{\mathbb R^d}\bigg[\int_{\mathbb R^d}e^{-{\rm i}c_s\omega\hat{x}\cdot(y-z)}\mathscr R ({\bm v}_{p,j}{\bm v}_{p,l}^{\top})\cdot\mathscr R K_{\bm f}^c(y,z)dz\bigg]e^{-{\rm i}c_s\tau\hat{x}\cdot y}dy\\
&&=|C_d|^2c_p^4c_s^{d-3}(\omega+\tau)^{\frac{d-3}2}\omega^{\frac{d-3}2}\bigg[\int_{\mathbb R^d}\mathscr R ({\bm v}_{p,j}{\bm v}_{p,l}^{\top})\cdot\mathscr R A^c(y)e^{-{\rm i}c_s\tau\hat{x}\cdot y}dy|c_s\omega\hat{x}|^{-m}+O(\omega^{-m-1})\bigg]\\
&&=|C_d|^2c_p^4c_s^{d-3-m}\left(\frac{\omega}{\omega+\tau}\right)^{\frac{3-d}2}\mathscr R ({\bm v}_{p,j}{\bm v}_{p,l}^{\top})\cdot\mathscr R \widehat{A^c}(c_s\tau\hat{x})\omega^{d-3-m}+O(\omega^{d-4-m}),
\end{eqnarray*}
which leads to 
\begin{eqnarray}\nonumber
&&\lim_{Q\to\infty}\frac{1}{Q}\int_Q^{2Q}\omega^{m+3-d}c_s^{m+3-d}c_p^{-4}{\mathbb E}\left[u_{p,j}^{\infty}\left(\hat{x},c_sc_p^{-1}(\omega+\tau)\right)\overline{u_{p,l}^{\infty}\left(\hat{x},c_sc_p^{-1}\omega\right)}\right]d\omega\\\label{e12}
&&\qquad\qquad=|C_d|^2\mathscr R ({\bm v}_{p,j}{\bm v}_{p,l}^{\top})\cdot\mathscr R \widehat{A^c}(c_s\tau\hat{x}).
\end{eqnarray}
Similarly, by noting that
\begin{eqnarray*}
&&{\mathbb E}\left[u_{p,j}^{\infty}\left(\hat{x},c_sc_p^{-1}(\omega+\tau)\right)\overline{u_{s,l}^{\infty}\left(\hat{x},\omega\right)}\right]\\
&&=|C_d|^2c_p^2c_s^{d-1}(\omega+\tau)^{\frac{d-3}2}\omega^{\frac{d-3}2}\int_{\mathbb R^d}\int_{\mathbb R^d}e^{-{\rm i}c_s(\omega+\tau)\hat{x}\cdot y}e^{{\rm i}c_s\omega\hat{x}\cdot z}{\mathbb E}\left[{\bm v}_{p,j}\cdot {\bm f(y)} {\bm v}_{s,l}\cdot\overline{{\bm f}(z)}\right]dydz\\
&&=|C_d|^2c_p^2c_s^{d-1}(\omega+\tau)^{\frac{d-3}2}\omega^{\frac{d-3}2}\bigg[\int_{\mathbb R^d}\mathscr R ({\bm v}_{p,j}{\bm v}_{s,l}^{\top})\cdot\mathscr R A^c(y)e^{-{\rm i}c_s\tau\hat{x}\cdot y}dy|c_s\omega\hat{x}|^{-m}+O(\omega^{-m-1})\bigg]\\
&&=|C_d|^2c_p^2c_s^{d-1-m}\left(\frac{\omega}{\omega+\tau}\right)^{\frac{3-d}2}\mathscr R ({\bm v}_{p,j}{\bm v}_{s,l}^{\top})\cdot\mathscr R \widehat{A^c}(c_s\tau\hat{x})\omega^{d-3-m}+O(\omega^{d-4-m}),
\end{eqnarray*}
we have
\begin{eqnarray}\nonumber
&&\lim_{Q\to\infty}\frac{1}{Q}\int_Q^{2Q}\omega^{m+3-d}c_s^{m+1-d}c_p^{-2}{\mathbb E}\left[u_{p,j}^{\infty}\left(\hat{x},c_sc_p^{-1}(\omega+\tau)\right)\overline{u_{s,l}^{\infty}\left(\hat{x},\omega\right)}\right]d\omega\\\label{e13}
&&\qquad=|C_d|^2\mathscr R ({\bm v}_{p,j}{\bm v}_{s,l}^{\top})\cdot\mathscr R \widehat{A^c}(c_s\tau\hat{x}).
\end{eqnarray}

Following the same procedure as above, we may get the limit of the following correlations:
\begin{eqnarray}\nonumber
&&\lim_{Q\to\infty}\frac{1}{Q}\int_Q^{2Q}\omega^{m+3-d}c_s^{m+1-d}c_p^{-2}{\mathbb E}\left[u_{s,j}^{\infty}\left(\hat{x},\omega+\tau\right)\overline{u_{p,l}^{\infty}\left(\hat{x},c_sc_p^{-1}\omega\right)}\right]d\omega\\\label{e14}
&&\qquad=|C_d|^2\mathscr R ({\bm v}_{s,j}{\bm v}_{p,l}^{\top})\cdot\mathscr R \widehat{A^c}(c_s\tau\hat{x}),\\\nonumber
&&\lim_{Q\to\infty}\frac{1}{Q}\int_Q^{2Q}\omega^{m+3-d}c_s^{m-1-d}{\mathbb E}\left[u_{s,j}^{\infty}\left(\hat{x},\omega+\tau\right)\overline{u_{s,l}^{\infty}\left(\hat{x},\omega\right)}\right]d\omega\\\label{e15}
&&\qquad=|C_d|^2\mathscr R ({\bm v}_{s,j}{\bm v}_{s,l}^{\top})\cdot\mathscr R \widehat{A^c}(c_s\tau\hat{x}).
\end{eqnarray}
Note that the coefficients in \eqref{e12}--\eqref{e15} satisfy
\begin{eqnarray*}
&&\mathscr R \left({\bm v}_{p,j}{\bm v}_{p,l}^{\top}+{\bm v}_{p,j}{\bm v}_{s,l}^{\top}+{\bm v}_{s,j}{\bm v}_{p,l}^{\top}+{\bm v}_{s,j}{\bm v}_{s,l}^{\top}\right)\\
&=&\mathscr R \left({\bm v}_{p,j}{\bm v}_{p,l}^{\top}+{\bm v}_{p,j}{\bm v}_{s,l}^{\top}+({\bm e}_j-{\bm v}_{p,j}){\bm v}_{p,l}^{\top}+({\bm e}_j-{\bm v}_{p,j}){\bm v}_{s,l}^{\top}\right)\\
&=&\mathscr R \left({\bm e}_j{\bm v}_{p,l}^{\top}+{\bm e}_j(\bm e_l-{\bm v}_{p,l})^{\top}\right)\\
&=&\mathscr R ({\bm e}_j{\bm e}_l^{\top}),
\end{eqnarray*}
which yields
\begin{eqnarray*}
\mathscr R ({\bm v}_{p,j}{\bm v}_{p,l}^{\top}+{\bm v}_{s,j}{\bm v}_{p,l}^{\top}+{\bm v}_{s,j}{\bm v}_{s,l}^{\top}+{\bm v}_{p,j}{\bm v}_{s,l}^{\top})\cdot \mathscr R \widehat{A^c}(c_s\tau\hat{x})=\widehat{a_{jl}^c}(c_s\tau\hat{x}).
\end{eqnarray*}
Adding (\ref{e12})--(\ref{e15}), we derive that 
\begin{align*}\nonumber
\lim_{Q\to\infty}\frac{1}{Q}\int_Q^{2Q}&\omega^{m+3-d}c_s^{m+3-d}\bigg\{c_p^{-4}{\mathbb E}\left[u_{p,j}^{\infty}\left(\hat{x},c_sc_p^{-1}(\omega+\tau)\right)\overline{u_{p,l}^{\infty}\left(\hat{x},c_sc_p^{-1}\omega\right)}\right]\\\nonumber
&+c_s^{-2}c_p^{-2}{\mathbb E}\left[u_{p,j}^{\infty}\left(\hat{x},c_sc_p^{-1}(\omega+\tau)\right)\overline{u_{s,l}^{\infty}(\hat{x},\omega)}\right]\\
&+c_s^{-2}c_p^{-2}{\mathbb E}\left[u_{s,j}^{\infty}\left(\hat{x},\omega+\tau\right)\overline{u_{p,l}^{\infty}\left(\hat{x},c_sc_p^{-1}\omega\right)}\right]\\
&+c_s^{-4}{\mathbb E}\left[u_{s,j}^{\infty}\left(\hat{x},\omega+\tau\right)\overline{u_{s,l}^{\infty}\left(\hat{x},\omega\right)}\right]\bigg\}d\omega=|C_d|^2\widehat{a_{jl}^c}(c_s\tau\hat{x})
\end{align*}
for $j,l=1,\cdots,d$, which can be rewritten into a compact form
\begin{align}\nonumber
\lim_{Q\to\infty}\frac{1}{Q}\int_Q^{2Q}&\omega^{m+3-d}c_s^{m+3-d}\bigg\{c_p^{-4}{\mathbb E}\left[\bm u_{p}^{\infty}\left(\hat{x},c_sc_p^{-1}(\omega+\tau)\right)\overline{\bm u_{p}^{\infty}\left(\hat{x},c_sc_p^{-1}\omega\right)}^\top\right]\\\nonumber
&+c_s^{-2}c_p^{-2}{\mathbb E}\left[\bm u_{p}^{\infty}\left(\hat{x},c_sc_p^{-1}(\omega+\tau)\right)\overline{\bm u_{s}^{\infty}(\hat{x},\omega)}^\top\right]\\\nonumber
&+c_s^{-2}c_p^{-2}{\mathbb E}\left[\bm u_{s}^{\infty}\left(\hat{x},\omega+\tau\right)\overline{\bm u_{p}^{\infty}\left(\hat{x},c_sc_p^{-1}\omega\right)}^\top\right]\\\label{e16}
&+c_s^{-4}{\mathbb E}\left[\bm u_{s}^{\infty}\left(\hat{x},\omega+\tau\right)\overline{\bm u_{s}^{\infty}\left(\hat{x},\omega\right)}^\top\right]\bigg\}d\omega
=|C_d|^2\widehat{A^c}(c_s\tau\hat{x}).
\end{align}

A similar result for the strength $A^r$ of the the relation operator can be obtained based on the same procedure:
\begin{align}\nonumber
\lim_{Q\to\infty}\frac{1}{Q}\int_Q^{2Q}&\omega^{m+3-d}c_s^{m+3-d}\bigg\{c_p^{-4}{\mathbb E}\left[\bm u_{p}^{\infty}\left(\hat{x},c_sc_p^{-1}(\omega+\tau)\right)\bm u_{p}^{\infty}\left(-\hat{x},c_sc_p^{-1}\omega\right)^\top\right]\\\nonumber
&+c_s^{-2}c_p^{-2}{\mathbb E}\left[\bm u_{p}^{\infty}\left(\hat{x},c_sc_p^{-1}(\omega+\tau)\right)\bm u_{s}^{\infty}(-\hat{x},\omega)^\top\right]\\\nonumber
&+c_s^{-2}c_p^{-2}{\mathbb E}\left[\bm u_{s}^{\infty}\left(\hat{x},\omega+\tau\right)\bm u_{p}^{\infty}\left(-\hat{x},c_sc_p^{-1}\omega\right)^\top\right]\\\label{e17}
&+c_s^{-4}{\mathbb E}\left[\bm u_{s}^{\infty}\left(\hat{x},\omega+\tau\right)\bm u_{s}^{\infty}\left(-\hat{x},\omega\right)^\top\right]\bigg\}d\omega=C_d^2\widehat{A^r}(c_s\tau\hat{x}).
\end{align}

The above recovery formulas \eqref{e16} and \eqref{e17}, where the expectation of the correlations between $\bm u_p^\infty$ and $\bm u_s^\infty$ is involved, can be improved by removing the expectation. In fact, according to (\ref{e9})--(\ref{e10}), one can easily find that components $u_{p,j}^{\infty}$ and $u_{s,j}^{\infty}$ are both linear combinations of far-field patterns for acoustic waves given in (\ref{b6}) perturbed by random sources $f_i,$ $i=1,\cdots,d$, which are components of the source $\bm f$. Thus, the estimates of the far-field pattern for acoustic waves given in  Lemma \ref{lem2} also hold for $\bm u_p^\infty$ and $\bm u_s^\infty$. Then following the same procedure used in Theorem \ref{thm3}, we can get that the strengths $A^c$ and $A^r$ can be uniquely recovered from the compressional and shear far-field patterns at a single realization of the random source almost surely, which is stated in the following theorem.

\begin{theorem}
Let ${\bm f}$ satisfy Assumption \ref{as2} with $m\in (d-4,d]$. Then for all $\hat{x}\in {\mathbb S}^{d-1}$ and $\tau\geq 0$, it holds almost surely that 
\begin{align}\nonumber
\lim_{Q\to\infty}\frac{1}{Q}\int_Q^{2Q}&c_s^{m+3-d}\omega^{m+3-d}\bigg[c_p^{-4}\bm u_{p}^{\infty}\left(\hat{x},c_sc_p^{-1}(\omega+\tau)\right)\overline{\bm u_{p}^{\infty}\left(\hat{x},c_sc_p^{-1}\omega\right)}^\top\\\nonumber
&+c_s^{-2}c_p^{-2}\bm u_{p}^{\infty}\left(\hat{x},c_sc_p^{-1}(\omega+\tau)\right)\overline{\bm u_{s}^{\infty}(\hat{x},\omega)}^\top\\\nonumber
&+c_s^{-2}c_p^{-2}\bm u_{s}^{\infty}\left(\hat{x},\omega+\tau\right)\overline{\bm u_{p}^{\infty}\left(\hat{x},c_sc_p^{-1}\omega\right)}^\top\\\label{e18}
&+c_s^{-4}\bm u_{s}^{\infty}\left(\hat{x},\omega+\tau\right)\overline{\bm u_{s}^{\infty}\left(\hat{x},\omega\right)}^\top\bigg]d\omega=|C_d|^2\widehat{A^c}(c_s\tau\hat{x}),\\\nonumber
\lim_{Q\to\infty}\frac{1}{Q}\int_Q^{2Q}&c_s^{m+3-d}\omega^{m+3-d}\bigg[c_p^{-4}\bm u_{p}^{\infty}\left(\hat{x},c_sc_p^{-1}(\omega+\tau)\right)\bm u_{p}^{\infty}\left(-\hat{x},c_sc_p^{-1}\omega\right)^\top\\\nonumber
&+c_s^{-2}c_p^{-2}\bm u_{p}^{\infty}\left(\hat{x},c_sc_p^{-1}(\omega+\tau)\right)\bm u_{s}^{\infty}(-\hat{x},\omega)^\top\\\nonumber
&+c_s^{-2}c_p^{-2}\bm u_{s}^{\infty}\left(\hat{x},\omega+\tau\right)\bm u_{p}^{\infty}\left(-\hat{x},c_sc_p^{-1}\omega\right)^\top\\\label{e19}
&+c_s^{-4}\bm u_{s}^{\infty}\left(\hat{x},\omega+\tau\right)\bm u_{s}^{\infty}\left(-\hat{x},\omega\right)^\top\bigg]d\omega=C_d^2\widehat{A^r}(c_s\tau\hat{x}).
\end{align}
Moreover, the strengths $A^c$ and $A^r$ can be uniquely determined by (\ref{e18}) and (\ref{e19}), respectively, with $(\tau, \hat{x})\in\Theta$ and $\Theta\subset\mathbb R_+\times\mathbb S^{d-1}$ being any open domain.
\end{theorem}

\section{Conclusion}

In this paper, we have discussed the direct and inverse random source problems for acoustic waves, biharmonic waves, electromagnetic waves, and elastic waves. The source is assumed to be a centered GMIG random field whose covariance and relation operators are classical pseudo-differential operators. For such a rough source, the unique solvability is achieved for a larger class of distributions compared with the existing results. The inverse problem is to recover the principal symbols of the covariance and relation operators. A relationship is established in the high frequency limit which connects the principal symbols of the covariance and relation operators and the far-field pattern averaged over the frequency band generated from a single realization of the random source. Based on the relationship, the uniqueness of the inverse problem is obtained. 

A possible continuation of this work is to study the stochastic wave equations with a random potential, where both the source and the potential are complex-valued GMIG random fields. These problems are more challenging due to the nonlinearity and coupling of the random source and potential. We hope to be able to report the progress on these problems elsewhere in the future.

\end{document}